\documentclass[11pt]{amsart}
\usepackage{amssymb,latexsym,graphicx, amscd}
\newtheorem{theorem}{Theorem}[section]
\newtheorem{prop}[theorem]{Proposition}
\newtheorem{lemma}[theorem]{Lemma}
\newtheorem{remark}[theorem]{Remark}
\newtheorem{question}[theorem]{Question}
\newtheorem{definition}[theorem]{Definition}
\newtheorem{cor}[theorem]{Corollary}

\begin{document}

\title{Almost K\"ahler forms on rational $4-$manifolds}
\author{Tian-Jun Li }

\address{School  of Mathematics\\  University of Minnesota\\ Minneapolis, MN
55455}
\email{tjli@math.umn.edu}
\author{Weiyi Zhang}
\address{Mathematics Institute\\  University of Warwick\\ Coventry, 
CV4 7AL\\ England}
\email{weiyi.zhang@warwick.ac.uk}

\begin{abstract} We study Nakai-Moishezon type question and Donaldson's ``tamed to compatible" question for almost complex structures on rational four manifolds. By extending Taubes' 
subvarieties--current--form  technique to $J-$nef genus $0$ classes, we give affirmative answers of these two questions for all tamed almost complex structures on $S^2$ bundles over $S^2$ as well as for many geometrically interesting tamed almost complex structures  on other rational four manifolds, including the del Pezzo ones.
\end{abstract}
 \maketitle
\tableofcontents
\section{Introduction}
Let $M$ be a compact, oriented, 
smooth manifold.
An almost complex structure on $M$ is an endomorphism of $TM$ whose square is $-$Id.
An almost complex structure $J$ induces an involution on the space of 2-forms, $\Omega^2(M)$,
decomposing it as $\Omega_J^+\oplus \Omega_J^-
$. $J$ is said to be tamed if there is
a symplectic form $\omega$ such that the bilinear form $\omega(\cdot, J(\cdot))$ is positive definite. In this case, $\omega$ is called a
taming form of $J$, and we also say that $J$ is tamed by $\omega$.
A taming form of $J$ is said to be compatible with $J$ if it lies in $\Omega_J^+$.
$J$ is said to be almost K\"ahler if there is a compatible form.

In  \cite{D}, Donaldson raised the following  question:
\begin{question} \label{dlz}
Suppose $J$ is an almost complex structure on a compact, oriented, smooth $4-$manifold $M$.
If $J$ is tamed, is  $J$ almost K\"ahler?
\end{question}

The local version of the question was known  to be true in dimension 4,
but false in higher dimensions (\cite {Sik94, RT, Lej, Br}).
Donaldson suggested an approach via the symplectic Calabi-Yau equation, and progress on this equation has
been made by Weinkove, Tosatti and Yau (cf. \cite{TW-survey}).

 It was known to hold when 
$M=\mathbb{CP}^2$ due to deep works of Gromov \cite{Gr} and Taubes
\cite{T}.
In \cite{LZ}, we observed that this  is  true for any integrable $J$, 
and the same was shown
for homogeneous $J$ in \cite{LT}.

In the case $b^+(M)=1$, Taubes  has recently made remarkable progress  in \cite{T1}.  He answers Question \ref{dlz} affirmatively for generic tamed almost complex structures
 in this case. 

The main purpose of this paper is to study Question \ref{dlz} for rational $4-$manifolds.  Here a rational $4-$manifold
refers to one of the following smooth 4-manifolds:
 $\mathbb{CP}^2$, $S^2\times S^2$ and blow-ups of them.

Our  construction is particularly successful  for 
 $S^2-$bundle over $S^2$: 
\begin{theorem}\label{s2bundle-1}
 Any tamed $J$ on $S^2\times S^2$ or $\mathbb C \mathbb P^2\# \overline{\mathbb C \mathbb P^2}$ is almost K\"ahler. 

\end{theorem}

The remaining rational  manifolds are of the form $\mathbb C \mathbb P^2\# k\overline{\mathbb C \mathbb P^2}$ with $k\geq 2$. We   settle Donaldson's question for these manifolds 
under a simple condition on $-1$ curves. 

\begin{theorem} \label{k  curves} 
 Suppose $M=\mathbb C \mathbb P^2\# k\overline{\mathbb C \mathbb P^2}$ with $k\geq 2$ and 
  $J$ is  tamed. If there are  $k$  disjoint $-1$  curves, then $J$ is almost K\"ahler. 
\end{theorem}

To prove his genericity result, Taubes explores  close connections between pseudo-holomorphic subvarieties (see Definition \ref{subvariety}) and
almost K\"ahler forms in dimension $4$. In arbitrary  dimension, they have
positive pairings. A special feature in dimension $4$ is that they
both lie in the space of closed, non-negative,  $J-$invariant $2$-currents.
In particular,  Taubes  introduced   a distributional analogue of an almost K\"ahler form,  which we call a Taubes current (see Definition \ref{tcdef}).

Let $\omega$ be a symplectic form
on a compact, oriented, smooth $4-$manifold with $b^+=1$.  The basic strategy of Taubes is to first carefully pick  a smooth subfamily  of
evenly distributed irreducible $J-$holomorphic subvarieties in the
class of $N[\omega]$ for $N$ large, at least when $[\omega]$ is rational, then to obtain a Taubes
current via integration.  
 Finally, Taubes showed that such a
 current can be first smoothed to a $J-$tamed form with a dominating $J-$invariant part,  and then further adjusted to  a genuine almost K\"ahler form.

For a rational manifold, 
to avoid generic choices of $J$ in several places of Taubes'  construction, we apply  the subvarieties--current--form  
technique to classes of  genus zero smooth   subvarieties with positive self-intersection. However, in general  we could only hope 
 to  first construct a weaker version of Taubes current which  degenerates on a finite  union of  
 subvarieties with negative self-intersection.  Then we  
 try to sum several such weak Taubes currents  to get a honest Taubes current.  We call 
 Taubes current obtained this way a spherical Taubes current.

 As in \cite{T1}, while  the construction  of weak 
 Taubes current in  \ref{Scurrent} is via integration over 
  the irreducible subvariety part $\mathcal M_{irr}$ of the moduli space, we still need a good control  of the reducible subvariety part 
 $\mathcal M_{red}$. With this in mind, we establish   in \cite{LZrc}  a clean structural picture of reducible subvarieties  for a $J-$nef class with $J-$genus $0$ (see Theorem \ref{emb-comp}).  This is crucial for us to get rid of much of the ``genericity" assumption of \cite{T1}. 

For a more detailed summary of Taubes'  subvarieties--current--form technique
and our adaptation, 
see section \ref{s-c-f}.

\bigskip

 The subvarieties--current--form  technique is also useful to further determine  the almost K\"ahler  cone    
\begin{equation}\label{J symplectic cones} \begin{array}{lll}
\mathcal K_J^{c}&=&\{[\omega]\in H^2(M;\mathbb R)|\hbox{$\omega$  is compatible with
$J$}\}
\end{array}
\end{equation}
in terms of the curve cone.
The almost K\"ahler cone $\mathcal K_J^{c}$ is a convex cohomology cone
 contained in the positive cone $$\mathcal P=\{e\in H^2(M;\mathbb R)|e\cdot e >0\}.$$

The  curve  cone of an almost complex manifold $(M, J)$,  denoted by $A_J(M)$, is the convex  cone in $H_2(M, \mathbb R)$ generated by the set of homology  classes of  $J-$holomorphic subvarieties.
Let $A_J^{\vee, >0}(M)$ be the positive dual of $A_J(M)$ under the homology-cohomology pairing, and 
 set $$\mathcal P_J=A_J^{\vee, >0}(M) \cap \mathcal P.$$
 Clearly, $\mathcal K_J^c\subset A_J^{\vee, >0}(M)$ since the integral of an almost K\"ahler form over a $J-$holomorphic subvariety is positive.
Motivated by   the famous Nakai-Moishezon-Kleiman criterion in algebraic geometry which  characterizes  the ample cone  in terms of  the (closure of) curve cone for a 
projective $J$, 
and the recent  K\"ahler version  \footnote{Established by Buchdahl and Lamari in dimension 4, and by 
Demailly-Paun  in arbitrary dimension.} of the Nakai-Moishezon criterion (in dimension $4$),
 which  characterizes  the K\"ahler cone  in terms of  the curve cone for a 
K\"ahler  $J$, 
we
ask whether there is an almost K\"ahler version of the Nakai-Moishezon criterion.

\begin{question}  \label{aknm} \footnote{ This question makes sense for any 4-manifold, if we view $A_J(M)$
   as a convex cone of $H_{+}^J$ and define
 $A_J^{\vee, >0}(M)$ to be the positive dual of $A_J(M)$ under the pairing between $H_{+}^J$ and $H_{J}^+$ in \cite{LZ}.
 The question is:  Is  $\mathcal K_J^c$  a connected component of $\mathcal P_J$?
 See the survey paper \cite{DLZ2}.}
\label{dualC} Suppose  $M$ is a 
compact, oriented, smooth $4-$manifold with $b^+=1$ and $J$ is  almost K\"ahler.
 Is  the almost K\"ahler cone dual to the curve cone, i.e. $\mathcal K_J^c=\mathcal P_J$?
\end{question}

Via a detailed analysis of spherical Taubes currents, we are able to establish the almost K\"ahler Nakai-Moishezon criterion in the following two cases.

\begin{theorem} \label{s2bundle-2} The  almost K\"ahler Nakai-Moishezon criterion holds for any almost K\"ahler $J$ on $S^2\times S^2$ or $\mathbb C \mathbb P^2\# \overline{\mathbb C \mathbb P^2}$.  
  \end{theorem}

 \begin{theorem} \label{good}
The almost K\"ahler Nakai-Moishezon criterion holds for  any good almost K\"ahler $J$. \end{theorem}

In Theorem \ref{good}, a tamed $J$ on  a rational manifold is called {\it good} if (i)   there is a smooth
genus one  subvariety in the  anti-canonical class $-K_J$, and (ii) 
any irreducible  genus zero  subvariety of negative self-intersection is a $-1$ curve. 

Notice that, following from Theorems \ref{s2bundle-1} and \ref{k curves}, both Theorems \ref{s2bundle-2} and \ref{good} are still valid even if we assume $J$ is tamed instead of almost K\"ahler.

The organization of this paper is as follows.
In section 2 we review  properties of  moduli space of irreducible pseudo-holomorphic subvarieties as presented in \cite{T} and  introduce Taubes current. 
In section 3,  to illustrate some key features of our construction, 
 we construct spherical Taubes currents in the  line class of $\mathbb C \mathbb P^2$. The situation for the line class is simpler since there are no reducible rational curves. 
In section 4 we discuss various properties of  $J-$nef spherical class and show that there are plenty of pencils in the moduli spaces. 
In section 5 we combine the constructions in sections 3 and 4 to  
construct weak Taubes currents from a big $J-$nef spherical class. An immediate consequence is
the proof  of Theorems \ref{s2bundle-1} and  \ref{s2bundle-2} for $S^2-$bundles over $S^2$. 
We also study the geometry and 
combinatorics of  $K-$symplectic cone and $K-$sphere cone of $\mathbb C \mathbb P^2\# k\overline{\mathbb C \mathbb P^2}$ with $k\geq 2$ and  prove Theorems \ref{k curves}, \ref{good}  and \ref{sc}.
Finally, we compare our construction with the Kodaira embedding theorem.

 We appreciate discussions with T. Draghici, R. Gompf, Y. Ruan, C. Taubes, A. Tomassini, V. Tosatti, S. T. Yau and K. Zhu. 
 We are grateful  to the referee for careful reading of the manuscript and useful suggestions
 improving the presentation. 
 During the preparation of this work, the authors benefited from NSF grant 1065927 (of the first author). The second author is partially supported by AMS-Simons travel grant.

\section{Pseudo-holomorphic subvarieties and Taubes current}
Let $M$ be a closed, oriented $4-$manifold and $J$ be a tamed almost complex structure on $M$.
In this section  we summarize properties of $J-$holomorphic subvarieties following \cite{T1} and introduce Taubes current.

 We fix a symplectic form  $\omega$  tamed by $J$. Such a form defines a cohomology class $[\omega]$.
  Moreover,  
the polarization of the quadratic form given by $\omega(\cdot, J(\cdot))$ defines a $J-$invariant metric on $M$. 
Such a metric is used implicitly in all the follows to define distances on $M$, integration over open sets in Cartesian products of $M$, 
norms on (complexified) tensor bundles of $M$. 

\subsection{Pseudo-holomorphic subvarieties and their properties}
\begin{definition} \label{subvariety} A closed set $C\subset M$ with finite, nonzero 2-dimensional
Hausdorff measure is said to be a $J-$holomorphic subvariety if it has
no isolated points, and if the complement of a finite set of points
in $C$, called the singular points,  is a smooth submanifold with $J-$invariant tangent space.

\end{definition}

When $J$ is understood, we
will simply call a  $J-$holomorphic subvariety a subvariety. 
A subvariety is said to be smooth if it has no singular points. 

A  subvariety $C$ has a canonical
orientation, which is used to define integration of smooth $2-$forms on its smooth part. The resulting linear functional on
the space of $2-$forms   
defines a closed, non-negative current of type $(1,1)$. 
We  denote the associated homology class by $e_C$.

\subsubsection{Genus of an irreducible subvariety}
A subvariety is said to be irreducible if its smooth locus is connected. 
Any given subvariety is a union of a finite set of irreducible subvarieties. 

Suppose $C$ is an irreducible subvariety.
Then it is the image of a $J-$holomorphic map
$\phi:C_0\to  M$ from a complex connected curve $C_0$, where $\phi$
is an embedding off a finite set. $C_0$ is called the model curve and $\phi$ is called the tautological map. 
The map $\phi$ is uniquely determined up to automorphisms of $C_0$. 
This understood, the homology class $e_C$ is simply the push forward of the fundamental class of $C_0$ via $\phi$.

The genus of an irreducible subvariety $C$ is defined to be the genus of its model curve $C_0$. 
There is another type of genus associated to the class $e_C$ defined as follows.

Given a class $e$ in $H_2(M;\mathbb Z)$,
introduce  the $J-$genus of $e$, 
\begin{equation} \label{J-genus}
\begin{array}{lll}
 g_J(e)&=&\frac{1}{2}(e\cdot e+K_J\cdot e)+1,
 \end{array}
 \end{equation}
  where $K_J$ is the canonical class of $J$. Notice that $g_J(e)$ is called the virtual genus in
  many literature. 
  
If $e=e_C$ for some subvariety $C$, then $g_J(e)$ is non-negative.  In fact, by the adjunction inequality in \cite{McD},  $g_J(e)$ is bounded from below by the genus of the  model curve $C_0$ of $C$, with equality if and only if $C$ is smooth.

\subsubsection{The normal operator $D_C$}

Suppose $C$ is an irreducible subvariety and $C_0$ is its model curve with the tautological map $\phi$. 
Let us first introduce the normal bundle of $C$. 

If $p\in C_0$ is not a critical point of $\phi$, then $L_p=\phi_*(TC_0|_p)$ is a complex line in $TM|_{\phi(p)}$.
If $p_0$ is a critical point of $C_0$,  for $p$ in  a deleted neighborhood of $p_0$ consisting of 
non-critical points,
  $L_p\subset TM|_{\phi(p)} $ still converge to a line $L_{p_0}$ in $TM|_{\phi(p_0)}$. 
This follows 
from elliptic unique continuation. 
It also follows from the local canonical form of $\phi$ near $p_0$ (see e.g. (2.2) of \cite{T}):
 there is a complex coordinate $u$ for a disk in $C_0$ 
centered on a critical point $p_0$, and a complex coordinate $(z, w)$ centered 
on $\phi(p_0)$ such that $\phi$ has the form 
\begin{equation}\label{p0local}
  u\rightarrow (u^{n+1}+t_z, cu^{n+k+1}+t_w),
\end{equation} where $n, k\ge 1$ and $|t_z|\le c_0|u|^{n+2}, 
|dt_z|\le c_0|u|^{n+1}, |t_w|\le c_0|u|^{n+k+2}, |dt_w|\le c_0|u|^{n+k+2}$.
Hence $L$ extends over $p_0$ as the pull-back from 
$\mathbb C^2$ via (\ref{p0local}) of the span of the vector field 
$\frac{\partial}{\partial z}$.  Thus there is a  complex line bundle $N$, whose fiber over $p\in C_0$ 
is the quotient complex line  $TM|_{\phi(p)}/L_p$.
This complex line bundle over $C_0$ is called the normal  bundle of  $C$.  

Linearizing the equation $\bar{\partial}_J\phi=0$ for the $J-$holomorphic map $\phi: C_0\rightarrow M$, we obtain a $\mathbb R-$linear, differential operator 
$$
 D_{\phi}:C^{\infty}(C_0; \phi^*TM)\to C^{\infty}(C_0; \phi^*TM\otimes T^{0,1}C_0).
$$
Choose an almost Hermitian metric on $M$ to realize $N$ as a subbundle of $\phi^*TM$.
Then it induces a canonically associated  $\mathbb R-$linear, differential operator
\begin{equation}
 D_C:C^{\infty}(C_0; N)\to C^{\infty}(C_0; N\otimes T^{0,1}C_0).
\end{equation}
${D}_C$ is called the normal operator of $C$.  

Use a Hermitian metric on $C_0$ and the Levi-Civita connection on $M$ to define Sobolev completions of 
$C^{\infty}(C_0; N),  C^{\infty}(C_0; N\otimes T^{0,1}C_0)$. 
$D_C$ extends to  a bounded, Fredholm operator
from the Hilbert space $L_1^2(C_0; N)$ to  the Hilbert space $L^2(C_0; N\otimes T^{0,1}C_0)$.
Denote the index of this extension by  $d_C$. 
$d_C$ is always even, and is  bounded above by the even integer $2\iota_{e_C}$
 defined as follows. 

\begin{definition} Given a class $e$, introduce  its $J-$dimension, 
\begin{equation} \label{l}  \iota_e=\frac{1}{2}(e\cdot e-K_{J}\cdot e). 
\end{equation}
\end{definition}

Note also that  $d_C=2\iota_{e_C}$  if and only if $C$ is smooth.


\subsection{The moduli space}
In this subsection we fix a class $e$. 

The moduli space of subvariety in the class $e$,  $\mathcal M_e$,  
is defined as in \cite{T1}: Any element $\Theta$ in  $\mathcal M_e$  is a 
finite set of pairs, where each pair has the form $(C,m)$ with $C\subset M$ 
an irreducible subvariety and $m$ a positive integer. The set of pairs in 
an element is further constrained so that no two of its pairs have the same subvariety component, 
and so that $\sum me_C=e$.

\begin{definition}\label{eff}
A  homology class $e\in H_2(M; \mathbb Z)$ is said to be  
$J-$effective if $\mathcal M_e$ is nonempty. 
\end{definition}

\subsubsection{Topology}

Let $|\Theta|=\cup_{(C, m)\in \Theta}C$ denote the support of $\Theta$. Consider  the symmetric, non-negative function, $\varrho$, on $\mathcal M_e\times \mathcal M_e$ that is defined by the following rule:
\begin{equation} \varrho(\Theta, \Theta')=\sup _{z\in |\Theta|} \hbox{dist}(z, |\Theta'|)
+\sup _{z'\in |\Theta'|} \hbox{dist}(z', |\Theta|).
\end{equation}
The function $\varrho$ is used to measure distances on $\mathcal M_e$.

Given a smooth form $\nu$ introduce the pairing
$$(\nu, {\Theta} )=\sum_{(C, m)\in \Theta} m\int_{C}\nu.$$

The topology on $\mathcal M_e$ is defined in terms of convergent sequences:

A sequence $\{\Theta_k\}$ in $\mathcal M_e$ converges to a given element $\Theta$ if the following two conditions are met:

\begin{itemize}
\item  $\lim_{k\to \infty} \varrho (\Theta, \Theta_k)=0$.

\item  $\lim_{k\to \infty} (\nu, \Theta_k)=(\nu, \Theta)$ for any given smooth 2-form $\nu$.
\end{itemize}

Here is Proposition 3.1 in \cite{T1}.

\begin{prop} \label{compactness} The moduli space $\mathcal M_e$ is compact.
In particular, only finitely many classes are of the form $e_C$ with $(C, m)\in \Theta$ and $\Theta\in \mathcal M_e$.

\end{prop}


\subsubsection{Moduli spaces of irreducible subvarieties}
Fix an integer $h$.
We define  $\mathcal M_{h,e}\subset \mathcal M_e$ to be the subspace of irreducible  subvarieties of genus $h$.

With this understood, the following  is a direct consequence of the adjunction inequality. 

\begin{lemma}\label{maximal genus}
If $h=g_J(e)$ and $C\in \mathcal M_{h,e}$, then $C$ is smooth. 
\end{lemma}

Let $\Sigma$ be a smooth, oriented surface of genus $h$. Let $\mathfrak M_{h,e}$ be the space of somewhere 
injective $J$-holomorphic   maps  in the class $e$ and originated from $(\Sigma, j)$, where $j$ is
an arbitrary  complex structure on $\Sigma$.  
As a subset of the Fr\'echet space of smooth maps
from $\Sigma$ to $M$,  $\mathfrak M_{h,e}$ has a natural topology. 

Since every irreducible subvariety has a model curve, 
there is a  surjective map $\Psi$ from $\mathfrak M_{h,e}$
to $\mathcal M_{h,e}$.
A fundamental fact established in the appendix in \cite{T1}, whose proof is rather involved,  is that the topology on  $\mathcal M_{h,e}$
is the same as the induced one  from  $\mathfrak M_{h,e}$ via $\Psi$. More precisely, by Lemma A.13 in  \cite{T1},  at any $\phi:(\Sigma, j)\to M$ in  $\mathfrak M_{h,e}$, $\Psi$  is a local homeomorphism from $\mathfrak M_{h,e}$
to $\mathcal M_{h,e}$ when $h\geq 2$, and  in the case $h=0,1$, $\Psi$ is a local homeomorphism up to automorphisms of $(\Sigma, j)$.

With this understood, it follows  that well known topological properties of $\mathfrak{M}_{h,e}$ 
carry over to $\mathcal M_{h,e}$. 

\begin{theorem} (Propositions 3.2  and 3.3 in \cite{T1}) \label{manifold} There exists a smooth map, $f$, 
from a neighborhood of $0$ in $\ker D_C$ to  ${\rm coker}  D_C$; and  there exists a homeomorphism from $f^{-1}(0)$ to a
neighborhood of $C$ in $\mathcal M_{h,e}$ sending $0$ to $C$.

 The subset of $\mathcal M_{h,e}$ where the cokernel of $D_C$  is trivial has the structure of a smooth manifold of dimension $2\iota_e-2(g_J(e)-h) $;
 and the smooth structure is such that at any point in this set, the aforementioned homeomorphism from a neighborhood of $0$ in the kernel is
 a smooth embedding onto an open set.
 
  For   $J$ in a residual set in the space of almost complex structures,  ${\rm coker }  D_{(\cdot)}=0$ for each point  $\mathcal M_{h,e}$, and so the latter 
 has the structure of  a smooth manifold whose dimension is $2\iota_e-2(g_J(e)-h)$.
 \end{theorem}
 
 Concerning the vanishing of $\hbox{coker }  D_C$, we mention another  fact, 
 which is particularly useful in this paper. 
 Note that $D_C$ is a real Cauchy-Riemann operator on $N$. For such operators, there is the following automatic transversality result.

\begin{theorem}[\cite{HLS97}, \cite{IS99}]\label{auto-tran} Let $(\Sigma,j)$ be a Riemann surface of genus $h$, and $L$  a complex line bundle over $\Sigma$. 
Suppose $c_1(L)\geq 2h-1$.    Then ${\rm coker} D=0$ for any  real Cauchy-Riemann operator $D$ on $L$. 
\end{theorem}


\subsubsection{Moduli spaces with marked points}

Let $M^{[k]}$ denote the set of $k$ tuples of pairwise distinct points in $M$.
Given   $\Omega=(z_1, \cdots, z_k)$ in $M^{[k]}$, denote its support $\{z_1, \cdots, z_k\}$  in $M$ also by $\Omega$. 
Let $\mathcal M_{e}^{ \Omega}$ be the space of subvarieties in $\mathcal M_{e}$ 
passing through $\Omega$.

Notice  that  $\mathcal M_{e}^{ \Omega}$ is compact since it is a closed subset of the compact space $\mathcal M_e$. 

Given an integer $h$, let $\mathcal M_{h, e}^{ \Omega}=\mathcal M_{h,e}\cap \mathcal M_{e}^{ \Omega}$.

Suppose $h=g_J(e)$ and $C\in \mathcal M_{g_J(e),e}^{\Omega}$. Then  $C$ is smooth by Lemma \ref {maximal genus}, and so the normal bundle $N$ is a line bundle over $C$ itself. 
Consider the evaluation map at $\Omega$, $ev^{\Omega}: \Gamma(N)\to \oplus_{p\in \Omega}N|_p,$
and  the operator
$$D_C\oplus ev^{\Omega}: \Gamma(N)\to \Gamma(N\otimes T^{1, 0}C)\oplus (\oplus_{p\in \Omega}N|_p).$$

The index of $D_C\oplus ev^{\Omega}$ is $d_C-2k$. 
And the kernel of
$D_C\oplus ev^{\Omega}$ should be thought of as giving a sort of
Zariski tangent space to $ \mathcal M_{g_J(e),e}^{\Omega}$  at $C$ (as a point in the
space of smooth embeddings). 

The smooth subvariety $C$ is called $(J, \Omega)$
non-degenerate if the operator $D_C\oplus ev^{\Omega}$ has trivial
cokernel.
If this is the case, $\mathcal M_{h,e}^{ \Omega}$
is a smooth manifold of dimension $d_C-2k$  around $C$.

It is clear that
$D_C\oplus ev^{\Omega}$ has trivial
cokernel  if  $D_C$ has trivial cokernel and 
$ev^{\Omega}:\ker D_C\to \oplus_{p\in \Omega}N|_p$ is surjective.
In light of Theorem \ref{auto-tran}, it is more useful to test the $(J, \Omega)$
non-degeneracy via the real Cauchy-Riemann operator  $D_{C}^{\Omega}$ in \cite{B}, which we now describe. 
 
The  bundle $N$ over the complex curve  $C$ has a  natural holomorphic line bundle structure. 
Consider  the holomorphic line bundle $N(\Omega)$,  obtained from twisting $N$ by the divisor $-(z_1+\cdots +z_k)$. 
In \cite{B} Lemma 4, there is introduced the following operator
 $$D_{C}^{\Omega}:
C^{\infty}(C_0;N(\Omega))\longrightarrow C^{\infty}(C_0;N(\Omega)\otimes T^{0,1}C_0).$$ 
 $D_{C}^{\Omega}$ is also a real Cauchy-Riemann operator. 
Moreover,  there is an exact sequence, (15) in \cite{B},

\begin{equation}\label{Barraud}
0\to \ker D_{C}^{ \Omega}\to \ker D_C \to \oplus_{p\in \Omega}N|_p \to \hbox{coker} D_{C}^{ \Omega}\to \hbox{coker} D_C\to 0,
\end{equation}
where the middle map is $ev^{\Omega}$.
It follows from \eqref {Barraud}  that we have

\begin{lemma}\label{Barraud'}
$D_C\oplus ev^{\Omega}$ has trivial
cokernel if $D_{C}^{ \Omega}$ has trivial cokernel. 

\end{lemma}
\subsection{Subvarieties through a small ball}


\subsubsection{The local area bounds}

The following summarizes the local area bounds  of an irreducible  subvariety.

\begin{lemma}\label{2.2}
Let $J$ be a tamed almost complex structure. Fix a symplectic form $\omega$  on $M$ taming $J$ and the induced $J-$invariant metric. 
There exists $k\ge 1$, depending only on $J, \omega$, with the following significance: Let $C\subset M$ denote an irreducible subvariety intersecting $B_r(x)$. 
Fix $r>0$. 
Let $a_x(2r)$ denote the area of $C$'s intersection   with the ball of radius $2r$ in $M$ centered at $x$. Then $k^{-1}r^2<a_x(2r)<(e_C\cdot [\omega])kr^2$.
\end{lemma}

\begin{proof}
This is based on Lemma 2.2 in \cite{T1}  which states a similar bound with a constant $k'$ when $x$ is  a point in $C$.  

For the lower bound, notice that the intersection contains a radius $r$ ball centered at a point in $C$. Take $k_1=k'$. 

For the upper bound, notice that the intersection is contained in a radius $3r$ ball centered at a point in $C$. Take $k_2=9k'$. 

Thus $k=9k'$ is as required. 
\end{proof}


\subsubsection{Local structure around a smooth subvariety}
  To  describe the behavior
of  subvarieties in a neighborhood of a given point, it is useful to introduce a special sort of coordinate chart. Fix a point $x$ in $M$.  An adapted coordinate chart centered at $x$ denotes complex coordinates, $(z, w)$ defined on a radius $c_0^{-1}$ ball centered at $x$ with both vanishing at $x$, with $dz$ and
$dw$ orthonormal at $x$, with $\{dz, dw\}$ spanning $T^{1,0}M$ at $x$, and with the norms of $|\nabla dz|$ and $|\nabla dw|$ bounded on the coordinate domain by $c_0$.

Suppose $C$ is a smooth subvariety passing through $x$.  
Fix an adapted coordinate chart centered at $x$ so as to identify a neighborhood of $x$ in $M$ with a ball about the origin in $\mathbb C^2$. There exists $R>1$ such  that
$C$ appears in the radius $R^{-1}$ ball about the origin in $\mathbb C^2$ as the image from $\mathbb C$ to $\mathbb C^2$  that has the form
\begin{equation}\label{line form}
u\to \theta u +  \nu
\end{equation}
where $|\nu|< R|u|^2$ and $|d\nu|< R|u|$, and such that $\theta\in \mathbb C^2$ has norm $1$.

The following is Lemma 4.2 in \cite{T1}.

\begin{lemma}\label{curves in a ball}
Let $C$ be a smooth subvariety  of genus $h$ described as above in an adapted chart centered at $x\in C$. 

Fix $\epsilon>0$. There is a neighborhood of $C$ in $\mathcal M_{e_C, h}$  where each subvariety 
$C'$ intersects the ball of radius ${1\over 2}R^{-1}$ about the origin, and this intersection is the image of a map from
$\mathbb C$ to $\mathbb C^2$ of the form
\begin{equation}
u\to x'+\theta'u+\nu'
\end{equation}
with
\begin{itemize}
\item $|x'|<\epsilon$,

\item $|\nu'|<R|u|(|x'|+|u|)$ and $|d\nu'|<R(|x'|+|u|)$,

\item $\theta'\in \mathbb C^2$ is a unit vector with $|\theta'-\theta|<\epsilon$.

\end{itemize}
\end{lemma}


\subsubsection{The exponential map $\exp_C$}

Suppose   $C$ is smooth and coker $D_C=0$.  We describe a version of
exponential map in \cite{T} and \cite{T1} to identify a ball in ker $D_C$ with  
a neighborhood of $C$ in $\mathcal M_{e_C, g_J({e_C})}$.

Since $C$ is smooth, the normal bundle $N$ can be realized  as the orthogonal complement of $TC$ in $T_{1,0}M|_C$. 
In this case, there exists  a map, $\exp_C$, that is defined on a small radius disk bundle $N_1\subset N$ and has the following properties:
\begin{itemize}
\item $\exp_C$ maps the zero section to $C$; and its differential along zero section is an isomorphism from $TN|_0$ to $\phi^*T_{1,0}M$.

\item $\exp_C$ embeds each fiber of $N_1$ as a $J$-holomorphic disk in $M$.

\item dist$(\exp_C(v), C)\le K|v| $ where $K$ is a uniform constant independent of $v$ and $x\in C$ for any vector $v\in N$ with small norm $|v|$.
\end{itemize}
A construction of such a map is described in section 5d of \cite{T}. The last item above is essentially from Lemma 5.4 (3) there.

Let $\zeta$ denote a section of $N_1$. Then the image in $M$ of the map $\exp_C(\zeta(\cdot))$ is a $J-$holomorphic subvariety if and only if $\zeta$ 
obeys an equation of the form 
\begin{equation} \label{equation} D_C\zeta+\tau_1\partial \zeta+\tau_0=0.
\end{equation}
Here $\tau_1$ and $\tau_0$ are  smooth, fiber preserving maps from $N_1$ to Hom$(N\otimes T^{1,0}C; N\otimes T^{0,1}C)$ and to $N\otimes T^{0,1}C$
that obey $|\tau_1(b)|\leq c_0|b|$ and $|\tau_0(b)|\leq c_0|b|^2$.

Since $\ker D_C$ is of finite dimension, all the norms on it are equivalent. Choose any norm $|\cdot |$, e.g. the $L^2$ norm or the sup norm,  on $\ker D_C$.
The map $\exp_C$ can be used to identify a fixed radius ball  of $\ker D_C$ with a neighborhood of $C$ in $\mathcal M_{e_C, g_J({e_C})}$
(see Lemmas 4.6  and 4.9 in \cite{T1}):
Suppose $C$ is smooth and coker $D_C =0$. For  $\kappa$ sufficiently large,  
there is a diffeomorphism from the radius $\kappa^{-2}$ ball in $\ker D_C$ onto an open set in $\mathcal M_{e_C, g_J({e_C})}$
that contains the set of curves with $\varrho-$distance less than $\kappa^{-3}$ from $C$. 
Moreover, 
the map in question sends a given small normed vector $\eta\in \ker D_C$ to $\exp_C(\eta+\phi_C(\eta))$, 
where $\phi_{C}:\ker D_C\to C^{\infty}(C; N)$ is such that $\zeta=\eta+\phi_C(\eta)$ satisfies \eqref{equation} and 
any given $C^k$ norm of $\phi_C(\eta)$ is bounded by a multiple of $|\eta|^2$.

\subsection{Subvariety, Taubes current and almost K\"ahler form}
\subsubsection{Non-negative currents and forms}

Given an almost complex structure $J$ on $M$, it acts on $\Omega^1(M)$: $J\alpha(X)=-\alpha(JX)$.
The componentwise action of $J$ on $\Omega^2(M)$ is an involution, decomposing it into
$\Omega_J^+(M)\oplus \Omega_J^-(M)$, called $J-$invariant and $J-$anti-invariant parts respectively.
An $J-$invariant 2-form is said to be non-negative if it is non-negative
on any pair of tangent vectors $(v, Jv)$ in each point, and positive if the evaluation is
positive for any nonzero $v$. In particular, a compatible symplectic
form is a positive,  closed $J-$invariant form, and a tamed symplectic
form is a closed form whose $J-$invariant part is positive.

Non-negative  2-forms can be constructed from 1-forms in the following way: The complexification of
 $\Omega^1(M)$ decomposed as $\Omega^1(M)\otimes \mathbb C=\Omega^{1,0}(M)\oplus \Omega^{0,1}(M)$, where
$\Omega^{1,0}$ is the $i-$eigenspace of the extended $J-$action. For any $\sigma\in \Omega^{1,0}(M)$,
$i\sigma\wedge \bar \sigma$ is a non-negative form.

A 2-dimensional current is a bounded linear functional on the space of smooth 2-forms.
All currents here are understood to be 2-dimensional. A current is said to be closed if it annihilates $d\Omega^1(M)$,
 $J-$invariant, or type (1,1),  if it annihilates $\Omega_J^-(M)$.
A $J-$invariant  current is said to be non-negative if it is
non-negative on any non-negative $J-$invariant form, and
positive if the evaluation is positive whenever the non-negative $J-$invariant form is not identically zero.

If $C$ is a $J-$holomorphic subvariety, using its natural orientation, it defines a closed, non-negative $J-$invariant current
 via the integration on its
smooth part.

\subsubsection{Taubes current and regularization}
The following type of positive current was introduced in \cite{T1}, which we will refer to as a Taubes current.
Let $B_t(x)$ denote  the ball of radius $t$ and center $x$
 and
$f_{B_t(x)}$ denotes the characteristic function of $B_t(x)$. Remember that we choose an almost Hermitian metric on $(M, J)$ to define the balls. When $J$ is tamed by a symplectic form $\omega$, we use the polarization of the quadratic form given by $\omega(\cdot, J\cdot)$.
\begin{definition}\label{tcdef}
On an almost Hermitian $4-$manifold $(M,J, g)$, a closed positive $J-$invariant
current $T$ is called a Taubes current if there is a constant $k>1$ such that for any $x$ and small $t$, 
\begin{equation} \label{leqgeq}k^{-1}t^4<T(if_{B_t(x)}\sigma\wedge\bar\sigma)<kt^4.
\end{equation}
Here $\sigma$ denotes a point-wise unit
length
section of $T^{1,0}M|_{B_t(x)}$. 
\end{definition}

Since $M$ is compact, being a Taubes current is independent of the metric $g$. 
A Taubes current behaves like an almost K\"ahler form except it may not be smooth.
 The following observation is also due to Taubes (see the proof of Theorem 1 in \cite{T1}).

\begin{prop}\label{C-F} 
Suppose $M$ has $b^+=1$ and $J$ is a tamed almost complex structure on $M$. Given a Taubes current $T$, there is an almost K\"ahler form $\alpha$, s.t.  $[\alpha]=[T]$.
\end{prop}

This is proved by first  smoothing the Taubes current to a family of closed two forms $\Omega^{\epsilon}$ in a standard way. The property  \eqref{leqgeq} then ensures that, when $\epsilon$ is small,   $\Omega^{\epsilon}$ is non-degenerate, uniformly bounded and has dominate $J-$invariant part. Then the condition of $b^+=1$ is used to   kill the anti-invariant part by  $L^2$ method for small $\epsilon$,  keeping the two form  symplectic and in the same class. The second author generalizes this regularization result to all almost complex $4-$manifolds with a Taubes current \cite{Z2}.

\subsubsection{The subvarieties--current--form technique}\label{s-c-f}
Let $M$ be a compact, oriented, smooth $4-$manifold with 
$b^+=1$, and $J$ an almost complex structure on $M$ tamed by a symplectic form $\omega$. 
The basic approach of Taubes to Question \ref{dlz}
is to first carefully pick  a smooth subfamily  of
evenly distributed irreducible $J-$holomorphic subvarieties in the
class of $N[\omega]$ for $N$ large, at least when $[\omega]$ is rational, then to obtain a Taubes
current via integration.  
 Finally, Proposition \ref{C-F} provides a genuine almost K\"ahler form.

 Given $\omega$, let  
$\mathcal J_{\omega}$ be 
the space   of  almost complex structures tamed by $\omega$.
Taubes' main result in \cite{T1} affirms  Question \ref{dlz} for 
an open and dense 
 subset  of $\mathcal J_{\omega}$. This generic subset  is  constrained  by a sequence  of regularity properties on 
  the  space of irreducible  subvarieties in the  class $N[\omega]$, as well as by similar regularity assumptions on  various spaces of reducible subvarieties.

Let  $\mathcal J^t$ be the  Fr\'echet space  of tamed almost complex structures. Since  $\mathcal J^t$ is the union of
$\mathcal J_{\omega}$ over all symplectic forms, 
Taubes' result  immediately implies that Question \ref{dlz} has a positive answer  for an open and dense subset $\mathcal J_1$ of 
 $\mathcal J^t$. 
   If we stratify $\mathcal J^t$ according to  the set of  irreducible  subvarieties of negative self-intersection, $\mathcal J_1$ is  properly contained in  
the top stratum $\mathcal J_{top}$ defined as follows.  

 \begin{definition}\label{smallcone}
A tamed $J$ is in $\mathcal J_{top}$ if any irreducible $J-$holomorphic subvariety with negative self-intersection is a $-1$  curve. Here a $-1$ curve refers to a smooth genus zero subvariety with self-intersection $-1$. 
\end{definition}

 The main purpose of this paper is to study Question \ref{dlz} for rational manifolds 
 by applying the subvarieties--current--form  
technique to classes of  genus zero smooth   subvarieties with positive self-intersection.
The first reason to consider  the space of  such subvarieties is that    no genericity assumption is needed, namely, 
 we  always get a smooth family for any $J$ from the automatic transversality in \cite {HLS97}.  
 This makes it possible to settle Question \ref{dlz} completely for certain rational manifolds, where 
 such subvarieties exist.

 Given a tamed $J$ on a rational manifold $M$,  classes of  genus zero smooth   subvarieties with positive self-intersection can always be found inside  $S_{K_J}$, the set of classes represented by
 smoothly embedded spheres and having $J-$genus zero. 
 However, unlike the case of $N[\omega]$ in \cite{T1}, such a class is often only $J-$nef in the sense that the pairing with 
 an arbitrary   $J-$holomorphic subvariety in $M$  is non-negative but might vanish.  
Consequently,  in general  we could only hope 
 to  first construct 
 a weak spherical Taubes current,  which is a weaker version of Taubes current degenerating  on a finite  union of  
 subvarieties with negative self-intersection.  Then we try to find several   weak
 spherical Taubes currents 
 whose degeneracy loci  have empty intersection and sum them to get a honest Taubes current.  The 
 Taubes current obtained this way is called a spherical Taubes current.

 As in \cite{T1}, while  the construction  of weak spherical current in  \ref{Scurrent} is via integration over 
  the irreducible subvariety part $\mathcal M_{irr}$ of the moduli space, we still need a good control  of the reducible subvariety part 
 $\mathcal M_{red}$. 
 To achieve  this, 
 we establish   in \cite{LZrc}  the following  clean structural picture of reducible subvarieties  for a $J-$nef class in $S_{K_J}$:

 \begin{theorem} \label{emb-comp}  Suppose $e$ is a  $J-$nef class in $S_{K_J}$ and $\Theta$ is  a reducible subvariety in the class $e$.
 \begin{itemize}
\item 
If $\Theta$ is connected, then  
 each  irreducible component  of $\Theta$  is a smooth rational curve, and $\Theta$ is a tree configuration.
\item
If $J$ is tamed, then $\Theta$ is connected. 
\end{itemize}
 \end{theorem}
 
 It follows that  $\mathcal M_{red}$ is a finite union of Cartesian products of irreducible rational curve moduli spaces, which then has expected dimension (Proposition \ref{reducible-dim}). 
 In particular,   $\mathcal M_{red}$ has codimension at least one for {\it any} tamed almost complex structure (Proposition \ref{reducible-dim'}). 
 This is crucial for us to get rid of much of the ``genericity" assumption of \cite{T1}.

Another new feature in our construction is the existence of  plenty of pencils in the moduli space (Proposition \ref{gendirect}). This  linearity property of
the moduli space enables us to better  model slices of a small neighborhood of a subvariety in terms of subspaces of the tangent space, which in turns gives rise to the desired 
estimate of the volume of points lying in subvarieties through a tiny ball. 

Our  construction of spherical Taubes currents is particularly successful  for 
 $S^2-$bundle over $S^2$ as in Theorem \ref{s2bundle-1} and \ref{s2bundle-2}.
   Theorem \ref{k curves} is also fairly general. In particular, it applies to   the entire top stratum  $ \mathcal J_{top}$.
  
   Theorem \ref{k curves} also applies to 
     the set  $\mathcal J_{good}$ of good almost complex structures, which are defined after
     Theorem \ref{good}. 
  For $\mathbb C \mathbb P^2\# k\overline{\mathbb C \mathbb P^2}$ with $k\leq 9$, it is not hard to see that 
  $\mathcal J_{good}=\mathcal J_{top}$. Hence, the almost K\"ahler Nakai-Moishezon criterion (Theorem \ref{good}) holds for  any $J\in  \mathcal J_{top}$ in this case. But  if $k\geq 10$, $\mathcal J_{good}$ is a lower stratum.


\section{Line class Taubes current on $\mathbb C \mathbb P^2$}

In this section 
we begin with introducing $K_J-$spherical classes and discussing  various automatic regularity  properties
of  irreducible subvarieties in these classes.  
Then we construct a Taubes current from 
the space of lines in $\mathbb C \mathbb P^2$.
Although in this special case, there are no reducible rational curves,
it  illustrates many key features of the general  construction,   notably,   how the presence of  pencils 
 enables us to better  model  a small neighborhood of a subvariety in terms of the tangent space, which in turns gives rise to the 
order two estimate of the volume of points lying in lines through a tiny ball.


\subsection{$K_J-$spherical classes and smooth rational curves}
Let   $J$ be a tamed almost complex structure. 

\subsubsection{Subvarieties in a  $K_J-$spherical classes}

Let $S$ be the set of homology classes which are represented by smoothly embedded spheres.
 
The set of  $K_J-$spherical classes is defined to be  $$S_{K_J}=\{e\in S|g_J(e)=0\}.$$
The following is a consequence of Seiberg-Witten theory, see e.g.  \cite{LL1}. 

\begin{prop}\label{existence'}   
Suppose $e\in S_{K_J} $ with $e\cdot e\geq -1$. Then  for any symplectic form $\omega$ taming $J$, the  Gromov-Taubes  invariant 
of $e$ is nonzero. In particular, $\mathcal M_e$ is nonempty, i.e. $e$ is $J-$effective. 
\end{prop}

We use $\mathcal M_{irr, e}$  to denote the moduli space of irreducible subvarieties in class $e$. 
The following is an immediate consequence of the adjunction formula and the adjunction inequality. 

\begin{lemma} \label {irr=smooth} For $e\in S_{K_J}$, 

$\bullet$ $\iota_e=e\cdot e+1$, where $\iota_e$ is defined in  \eqref{l}.

$\bullet$ every element in  $\mathcal M_{irr,e}$  is a smooth rational curve. 
\end{lemma}

Thus for $e\in S_{K_J}$, $\mathcal M_{irr,e}$ is the same as  $\mathcal M_{0,e}$. 
Let $\mathcal M_{red,e}$ denote $\mathcal M_e \setminus \mathcal M_{irr,e}$. 

 Given a $k(\le {\iota_e})$ tuple of distinct points $\Omega$, recall that   $\mathcal M_e^\Omega$ is the space of subvarieties in $\mathcal M_e$  that contains all entries of $\Omega$.
  Introduce similarly $\mathcal M_{irr,e}^{\Omega}$ and $\mathcal M_{red,e}^{\Omega}$.  We will often drop the subscript $e$. 

Let $ S^+, S^0, \mathcal E\subset S$ be the subsets of positive square, square 0, square $-1$ classes respectively.
$ S^+$ is nonempty if and only if $M$ is a rational manifold (\cite{L}).
Let $S^+_{K_J}=S^+\cap S_{K_J}$ and define ${S}^0_{K_J}$, $\mathcal E_{K_J}$ similarly.

Let $k$ be an integer.  
Denote by $\mathcal M_{e, k}$ the subset in $\mathcal M_{e}\times M^{[k]}$ that
consists of elements of the form $(C,x_1,\cdots, x_k)$ with each $x_i\in C$.
Define $\mathcal M_{irr,e, k}$ and $\mathcal M_{red, e, k}$ similarly.

Consider the restrictions of the projection map to the $M^{[k]}$ factor,  
$\pi_k:\mathcal M_{e, k}\to M^{[k]}$,  $\pi_{irr, k}:\mathcal M_{irr,e, k}\to M^{[k]}$, $\pi_{red, k}:\mathcal M_{red, e, k}\to M^{[k]}$.

By Proposition \ref{existence'}, we have   

\begin{lemma}\label{existence} Suppose $e\in S^{+}_{K_J}\cup S^{0}_{K_J} \cup \mathcal E_{K_J}$.  For any $\Omega\in M^{[\iota_e]}$, $\mathcal M^{\Omega}_e$ is non-empty. In other words, 
 $\pi_{\iota_e}$ is surjective. 
\end{lemma}

\subsubsection{Smooth rational curves}
We assume now that  $e$ is a class  represented by a smooth  rational curve. 
In particular, $e\in S_{K_J}$.

 Introduce 
$$l_e=\max\{\iota_e, 0\}.$$

One special feature of the moduli space of smooth rational curves is the following automatic transversality, which is valid for an arbitrary tamed almost complex structure. 

\begin{lemma} \label{pdmanifold} Let $e$ be a class  represented by a smooth  rational curve with $e\cdot e\geq -1$.
Then $\mathcal M_{irr,e}$ is a smooth manifold of dimension $2l_e$.

Moreover, if we choose a set of $k\le l_e$ distinct points $\Omega\subset C$, where   $C\in \mathcal M_{irr,e}$, then 
$\mathcal M_{irr,e}^{\Omega}$  is a smooth manifold of dimension $2(l_e-k)$.

\end{lemma}
\begin{proof}  This observation is essentially contained in in Corollary 2 in \cite{Sik} (see also \cite{B}).

Suppose   $C\in \mathcal M_{irr,e}$. By Lemma \ref{irr=smooth}, $C$ is a smooth rational curve.  

Let $N$ be its normal bundle. 
 Notice that, since $c_1(N)=e\cdot e \geq -1$ and $g=0$, by Theorem \ref{auto-tran},  coker $D_C=0$. 
It follows from Theorem \ref{manifold} that  $\mathcal M_{irr,e}$ is a smooth manifold whose dimension is $d_C$. Since $C$ is smooth,  $d_C=2\iota_e$. Since $e\cdot e\geq -1$, $\iota_e\geq 0$ and
hence $l_e=\iota_e$. 

The second statement  is proved similarly.  By Lemma \ref{Barraud'}, it suffices to show that $D_{C}^{\Omega}$ has trivial cokernel. 
The proof is finished by noticing that $c_1(N(\Omega))=c_1(N)-k$ and the index  of $D_C\oplus ev^{\Omega}$ is $d_C-2k$. 

\end{proof}

Here is another  feature, specific to  rational curves.

\begin{lemma}\label{uniqueness} Let $e$ be a class  represented by a smooth  rational curve. 
\begin{itemize}
\item If $e\cdot e\leq -1$, then $\mathcal M_{irr,e}$ consists of a single element. 

\item If $e\cdot e\geq 0$, $l=l_e$,   and $(C, \Omega)\in \mathcal M_{irr, e, l}$, then  $\mathcal M_{irr, l}^{\Omega}$ consists of $C$ only. 
In other words, $\pi_{irr, l}$ is an injective  smooth map, and  
the image of $\pi_{irr, l}$ is disjoint from the image of  $\pi_{red, l}$. 

\item If $e\cdot e\geq 1$, $l=l_e$, the same uniqueness conclusion   is true when we impose constraints of  $l-1$ points and a complex direction at one of the $l-1$ points.
\end{itemize}
\end{lemma}
\begin{proof}
All the statements follow from  the positivity of intersections of distinct irreducible subvarieties. For the second bullet, notice that $e\cdot e < l$. For the third bullet, 
notice that  a tangency contributes at least 2 to the intersection number. 
\end{proof}

\begin{remark} There is an  analogous result to Propositions \ref{pdmanifold} involving tangency conditions. 
Suppose $C$ is smooth rational curve. 
If we choose  $k\le l$ distinct points $x_1,..., x_k$ in $C$ and  $k'<k$ with  $k+k'\le l$, then the set of smooth rational curves in $\mathcal M^{x_1, ..., x_k}_{irr,e_C}$   having the same tangent space at  the $k'$ points 
$x_1, ..., x_{k'}$ as $C$ is still a smooth manifold, whose dimension is $2(l-k-k')$.
For this case, the Zariski tangent space is given by the subspace of $\ker D_C^{x_1,..., x_k}$, 
   vanishing with order at least  two on $x_1, ..., x_{k'}$. Thus,
 we are considering the line bundle $N(2x_1+\cdots+2x_{k'}+x_{k'+1}+\cdots+x_k)$,  $N$ twisted by the divisor $-(2x_1+\cdots+2x_{k'}+x_{k'+1}+\cdots+x_k)$. In this case, the relevant Cauchy Riemann operator  is onto if $-K_J\cdot e>k+k'$, which is automatic since $  -K_J\cdot e=l+1$.

\end{remark}

 By Lemmas \ref{pdmanifold} and \ref{uniqueness} we have 
 
 \begin{prop} Let $e$ be a class  represented by a smooth  rational curve and $l=l_e$.
 Then $\mathcal M_{irr,e}$ has the structure of  a $2l$ dimensional manifold and 
 $\mathcal M_{irr,e, k}$ has the structure of  a $2l+2k$ dimensional  manifold.
 $\pi_{irr, k}: \mathcal M_{irr,e, k}\to M^{[k]}$ is a smooth map from  a $2l+2k$ dimensional  manifold
 to a $4k$ dimensional manifold.  
 \end{prop}

\subsection{Moduli space of  lines} \label{cp2b} 

Now, we assume $M=\mathbb C \mathbb P^2$. 
Let $H$ be the line class, namely, the generator of $H_2(M;\mathbb Z)$ such that $K_J=-3H$. 
Note that  $S_{K_J}=\{H, 2H\}$.

Let $\mathcal M$ denote $\mathcal M_{H}$, which we call the moduli space of lines. 
Notice that 
there are no reducible  curves in $\mathcal M$ so 
$\mathcal M_{irr}=\mathcal M$. Thus by Lemma \ref{irr=smooth}, every element in $\mathcal M$ is a smooth rational curve. 
Since  $l=l_H=2$,
by Lemma \ref{pdmanifold}, $\mathcal M$ is a compact, $4$ dimensional smooth manifold. In fact, it is diffeomorphic to $\mathbb C \mathbb P^2$ (see e.g. \cite{McK}).
 Since $\mathcal M$ is compact, by Lemma \ref{curves in a ball}, we have the following   (Lemma 4.12 in \cite{T1}).

\begin{lemma}\label{4.14}
There is a constant $s'<10^{-10}$ with the following significance: Fix a point $x$ and an adapted coordinate chart centered at $x$ so as to identify a neighborhood of $x$ in $M$ with a ball about the origin in $\mathbb C^2$. Let $C\in \mathcal M$ which intersects $B_{s'}(x)$. Then $C$ intersects the ball of radius $s'^2$ centered $x$ as the image of a map from a disk in $\mathbb C$ about the origin to $\mathbb C^2$ that has the form $u\to 
\theta u+\mathfrak r(u)$,  where 
$\theta\in \mathbb C^2$ has norm $1$ and where $\mathfrak r(u)$ is such that $|\mathfrak r(u)|\le s'^{-1} |u|^2$ and $|d\mathfrak r|_v|\le s'^{-1} |u|$. 

\end{lemma}

Let $s=s'^4$.

\subsubsection{Pencils of lines } \label{pencilcp2}

Fix a point $x_1\in M$. 
 By Lemma \ref{pdmanifold}, 
  $\mathcal M^{x_1}$ is 2-dimensional manifold. 
  In fact, $\mathcal M^{x_1}$ is a pencil, consisting  of  $\mathbb C \mathbb P^1$ family of lines.

 Choose  an orthonormal basis of vectors in $T_{x_1}M$ and use it to identify  the space  of complex directions at $x_1$
with $\mathbb C \mathbb P^1$. 
Consider   the smooth map $\tau^{x_1}:\mathcal M^{x_1}\to \mathbb C \mathbb P^1$
 by taking the tangent line  at $x_1$.  The claim is that there is a unique curve
through every direction at $x_1$. 
The uniqueness is a consequence of the positivity of intersections. The existence can be shown by
taking the limit of a sequence of `secant' curves $C_k$ through $x_1$ and $y_k$ with $y_k-x_1$ projecting to any given $\hat \theta\in \mathbb C \mathbb P^1$ and $y_k\to x_1$.


\subsubsection{Norms on  $\ker_{C}^{x_1}$}

Given a smooth curve $C$ 
in the pencil $\mathcal M^{x_1}$ with normal bundle $N$, the tangent space to $\mathcal M^ {x_1}$ at $C$
 can be identified with 
the vector space $\ker_{C}^{ x_1}\subset \Gamma(N)$ that consists of the sections in the kernel of $D_C$ 
that vanish at $x_1$. 

On this two dimensional space,  besides 
the  sup norm $\sup_C |\eta|$, 
  there are  other kinds of  norms due to the fact that $\mathcal M^{x_1}$ is a pencil.

\begin{lemma} \label{dist'} The following are norms on $\ker_{C}^{x_1}$:

$\bullet$ For $z\ne x_1$, the pointwise norm $|\eta(z)|$ of the vector $\eta(z)$ in $N|_z$.

$\bullet$ The pointwise norm of  $\tau_*^{x_1}|_C(\eta) $ as a vector in $T_{\hat \theta} \mathbb C \mathbb P^1$, where $\hat \theta$ denotes the direction $T_{x_1} C$ in $\mathbb C \mathbb P^1$.

For a fixed curve  $C$, these norms are equivalent to the sup norm  $\sup_C|\eta|$.
\end{lemma} 

\begin{proof}
Suppose $\eta$ is a non-trivial element in $\eta\in \ker_{C}^{ x_1}$.
Both claims rely on an observation in  \cite{HLS97} that  
 there is a new holomorphic structure on $N$ (depending on $\eta$)  with respect to which $\eta$ is  a holomorphic section.

 Since $c_1(N)=1$,
 due to  the positivity of intersection of holomorphic sections, $\eta$ only vanishes at $x_1$.  
 Especially, $\eta(z)\neq 0$. 
In other words, $|\eta(z)|$ is a norm of $\ker_{C}^{x_1}$ if $z$ is not  $x_1$.  

To establish the second claim, we need the following description of $\tau_*^{x_1}|_C$ of Taubes in part g) of the Appendix. 
Fix an adapted coordinate chart, $(z, w)$, centered at $x_1$ so that $C$ is tangent to the $w=0$ locus at the origin. 
The span of $\frac{\partial}{\partial w}$ is identified with the fiber of $N$ at $x_1$, as well as 
$T_{\hat \theta} \mathbb C \mathbb P^1$. 
Choose a holomorphic coordinate, $u$, for $C$ centered at $x_1$ with $du=dz$ at $x_1$. Then $\partial \eta|_{x_1}$,  when viewed as an element
in $T_{\hat \theta} \mathbb C \mathbb P^1$ using the identifications above, is the image of $\eta$ under the  differential of $\tau^{x_1}$ at $C$. 

From this description we see that $\tau_*^{x_1}|_C(\eta) =0$ only if $\eta$ has vanishing order at least 2 at $x_1$. Since $\eta$ is a holomorphic section
and $c_1(N)=1$, this is impossible. Thus we have shown that 
$\tau_*^{x_1}|_C:T_C\mathcal M^{x_1}\to T_{\hat \theta} \mathbb C \mathbb P^1$
is an isomorphism.

The last statement  is clear since  any two norms on a finite dimensional vector spaces are equivalent. 
\end{proof}

Notice that, since every $C\in \mathcal M^{x_1}$ is a smooth curve,  we have shown that  $\tau^{x_1}$ is a diffeomorphism.

\subsubsection{Lines meeting a small ball and order 2 estimates}

Given $x, x_1\in M$,  we consider the subset $\mathcal M^{x_1; B_t(x)}\subset \mathcal M^{x_1}$ whose element  intersects $B_t(x)$ for  small $t$. 
Denote in what follows the  line  through $x_1$ and $x$ by $C$.  

The following lemma is an analogue to Lemmas 4.5 and 4.9 in \cite{T1}.

\begin{lemma}\label{dist} 
Let   $x$, $x_1$ and $C$ be as stated as above. There is a constant $k_{\ref{dist}}>1$, depending only on  $s$,  ensuring the following inequalities for $\eta\in \ker_{C}^{x_1}$,
\begin{enumerate}
\item 
$\sup_C|\eta|\le k_{\ref{dist}}  |\eta(z)|$  if  $x_1$ is not in $B_s(x)$ and $z\in\overline{B_{\frac{s}{2}}(x)}$;
\item $\sup_C|\eta|\le k_{\ref{dist}} |\tau^{x_1}_{*} \eta|$  if  $x_1\in \overline{B_s(x)}$.

\end{enumerate}

\end{lemma}

\begin{proof} 
 The constants in Lemma \ref{dist'} can be chosen to be independent of $x$, $x_1$, $z$ since $M$, $M\setminus B_s(x)$ and $\overline{B_s(x)}$ 
are compact.
\end{proof}
\bigskip

Let $T^{x_1;B_t(x)}$ denote the set of points $x_2$ in $M$ that lies in a curve in $\mathcal M^{x_1}$ and intersecting $B_t(x)$.

\begin{lemma} \label{outside}
Suppose $x_1$ is not in $B_s(x)$. 
There are constants $k$ and $\kappa$ 
depending on $s$ with the following significance:
For $t<\kappa^{-3}$, 
the volume of $T^{x_1;B_t(x)}$  is bounded from above by $kt^2$. 

\end{lemma}

\begin{proof} Consider the unique  curve $C$  through $x_1$ and $x$.  
Since dist$(x_1, x)\geq s$, by Lemma \ref{dist'}, $|\eta(x)|$ is a norm on the $2-$dimensional vector space ker$_C^{x_1}$. 
Since $\mathcal M^{x_1}$ is a $2-$dimensional smooth manifold, apply the implicit function theorem as in Lemma 4.7 in \cite{T1}, we find there exists $\kappa>1$
with the following property: 
for $t< \kappa^{-3}$, there  is an embedding $\lambda_{C}^{ x_1}$ from the  
$|\eta(x)|\leq \kappa t$  disk in
$\ker_{C}^{ x_1}$ onto a neighborhood in 
$\mathcal M^{ x_1}$ of $C$, which brings $0$ to $C$, and 
contains $\mathcal M^{x_1; B_t(x)}$   as an open set in the image.
Since $\mathcal M^{x_1}$ and $M\setminus B_s(x)$ are compact,   $\kappa$ only depends on $s$. Moreover,  the image of $\eta\in \ker_{C}^{ x_1}$ via $\lambda_{C}^{ x_1}$ can be written as $\exp_C(\eta+\phi_{C}^{ x_1}(\eta))$ where $\phi_{C}^{x_1}(\cdot)$ maps the ball $\{\eta:|\eta(x)|\le 2 \kappa^{-1}\}$ smoothly into $\ker_C^{x_1}$, and  the $C^0$ norm of $\phi_{C}^{x_1}(\eta)$ is bounded by a multiple of $|\eta(x)|^2$.

Suppose  $x_2$ in $M$  lies in a curve $C'$ in $\mathcal M^{x_1}$ with  $C'\cap B_t(x)\ne \emptyset$.  Then 
there is   a vector $\eta\in \ker_{C}^{ x_1}$  with norm $|\eta(x)|\le \kappa t$, and 
 a point $p\in C$ such that $x_2=\lambda_{C}^{x_1}(\eta(p))$.

Notice that $ |\eta(p)+\phi_{C}^{ x_1}(\eta)(p)|\le |\eta(p)|+\kappa' |\eta(p)|^2$ by the estimate of $\phi_C^{x_1}$. 
 Hence, by the third bullet in 2.3.3, 
 $$\hbox{dist}(x_2, C)=\hbox{dist} (\lambda_C^{x_1}(\eta(p)), C)\le K |\eta(p)|.$$
 Further, by the first bullet of Lemma \ref{dist},
$$\hbox{dist}(x_2, C)\le K |\eta(p)| \le K\sup_C|\eta|
\le K\kappa_{\ref{dist}}|\eta(x)| \le 2K\kappa_{\ref{dist}}\kappa t.
$$

This implies that $x_2$ is constrained so as to lie
in a tubular neighborhood of $C$ whose radius is bounded above by $k'_2t$ (with $k_2'=2K\kappa_{\ref{dist}}\kappa$).

The area of $C$ is bounded by $k_0 H\cdot [\omega]$. 
So the volume of the points in $M$ that lie on a radius $t$ tubular neighborhood of any $C$ is bounded from above by $kt^2$. 
\end{proof}

\begin{lemma}\label{angle-volume}

Suppose  $x_1$ is in $\overline{B_s(x)}\setminus B_{Rt}(x)$. There is a constant $k$ depending on $s$ and $R$ with the following significance:

The volume of points $x_2$ in $M$ lying in a curve in $\mathcal M^{x_1}$ and intersecting $B_t(x)$ is bounded from above by $k\frac{t^2}{d^2}$, where $d={\rm dist}(x_1, x)$.

\end{lemma}

\begin{proof}

Introduce $\mathcal O^{x_1}_{\hat \theta, \delta}\subset \mathcal M^{x_1}$ to denote the set of curves that  mapped under $\tau^{x_1}$ to the disk in $\mathbb C \mathbb P^1$ of radius $\delta$ centered on the point $\hat \theta \in \mathbb C \mathbb P^1$. 

The following geometric consideration is (4.14) in \cite{T1}.
If $d>10^4$ and $t <\frac{s}{d}$, then the set of complex 1 dimensional lines that intersects $B_t(x)$ with $x=(d, 0)$ 
is contained in a disk in $\mathbb C \mathbb P^1$ of radius less than $c_0\frac{t}{d}$ with center the image of $(1, 0)$ of
$\mathbb C \mathbb P^1$.

With this understood, by Lemma \ref{4.14},  $\mathcal M^{x_1; B_t(x)}$ lies in $\mathcal O^{x_1}_{\hat \theta, \delta}$ with $\hat \theta$ here denoting the image of the point $(1, 0)$ and with $\delta<c \frac{t}{d}$.

   By Lemma  \ref{dist}(2), $\mathcal O^{x_1}_{\hat \theta, \delta}$ is contained in a tube of radius $k\delta$, whose  volume  is bounded from above by $k\delta^2$. Thus we have the desired volume estimate. 
\end{proof}

\subsection{Spherical Taubes current from the line class}

Let $\mathcal M_2$ be the moduli space of lines with two distinct points,
$$\mathcal M_2=\{(C, x_1, x_2)|C\in \mathcal M, x_i\in C\}\subset \mathcal M\times M^{[2]}.$$
Use $\pi_2$ to denote the projection map $\mathcal M_2\to M^{[2]}$. 
Let $\pi_{\mathcal M}$ be the projection map $\mathcal M_2\to \mathcal M$.

The portion of marked moduli space we need is  $\mathcal M_2^r$ for $0<r<\frac{s}{10}$, 
subject to the constraint $d(x_1, x_2)\geq r$. Notice that the image $\pi_{\mathcal M}(\mathcal M_2^r)$ is  $\mathcal M$ if $r$ is chosen that small.

For  $\eta \in \pi_2(\mathcal M^r_2)$, we introduce the non-negative, closed, invariant current $\phi_{\eta}$. It is defined by  $\phi_{\eta}(v)=\int_{C_{\eta}}v$, where $C_{\eta}$ is the unique line  $\pi_2^{-1}(\eta)$ and $v$ is a 2-form on $M$. Then we have the following {\it spherical current} $\Phi_H$ in the line class given by 
$$\Phi_H(v)=\int_{\eta\in \pi_2(\mathcal M^r_2)} \phi_{\eta}(v).$$

This current $\Phi_H(v)$ clearly satisfies Proposition 1.2 in \cite{T1}. Especially, it is a non-trivial, closed, non-negative $J-$invariant current on $\mathbb C \mathbb P^2$. In the rest of the section, we will prove that it is indeed a Taubes current.

\bigskip

\subsubsection{Upper bound}

Fix a smooth, non-increasing function $\chi:[0, \infty)\to [0,1]$ with value $1$ on $[0, \frac{1}{4}]$
and value $0$ on $[\frac{1}{2}, \infty)$. Use $\chi_t$ to denote the function $\chi(t^{-1}|\cdot|)$ on $\mathbb C^2$. 

\begin{prop} \label{upper'}
 The current $\Phi_H$ satisfies the upper bound in \eqref{leqgeq}.
\end{prop}

\begin{proof}

Let $s$ be as in Lemma \ref{4.14}. 
Fix $x\in M$ and adapted coordinates $(z, w)$ centered at $x$ with radius $s$. 
Let $t<10^{-5}  s$.
As in \cite{T1}, we only need to prove  $\Phi_H(i\chi_tdz\wedge d\bar z)<k t^4$. Moreover, $\Phi_H(i\chi_tdz\wedge d\bar z)$ is no greater than
$$k_0\int_{\eta=(x_1, x_2)\in \pi_2(\mathcal M_2^r)}(\int_{ C_{\eta}}\chi_t\omega).
$$

Notice that  $\int_{ C_{\eta}}\chi_t\omega=0$
if  $C_{\eta}\cap B_t(x)= \emptyset$.   
If  $C_{\eta}\cap B_t(x)\neq \emptyset$, then $C_{\eta}\cap B_t(x)$ is contained in a ball
of radius of $2t$ centered at some point in $B_t(x)$. 
By Lemma \ref{2.2},   the integrand $\int_{ C_{\eta}}\chi_t\omega$ is bounded by $H\cdot [\omega] k t^2$.

Thus it suffices to prove that the volume of the set
\begin{equation} \label{volume of pairs'} \{\eta=(x_1, x_2)\in \pi_2(\mathcal M_2^r)| C_{\eta}\cap B_t(x)\neq \emptyset\}
\end{equation}
is $O(t^2)$. 

We  follow  \cite{T1} to divide into three cases depending on the position of  $x_1$.

I. The first case is that $x_1$ is away from $B_s(x)$.  The upper bound for the choice of $x_1$ is Vol$(M)$.
Now we estimate the possible choices of $x_2$ for a fixed $x_1$.
It follows from Lemma \ref{outside} that the volume of $x_2$ is $k_2t^2$.
The factors   Vol$(M)$ and $k_2t^2$ multiply to an upper bound of $O(t^2)$ for the volume of the subset in \eqref{volume of pairs'} with $x_1\in M\setminus B_s(x)$. 
This bound depends on $s$.

II. The second case is when $x_1$ is in $B_s(x)$ but outside $B_{Rt}(x)$, where  $R=10^5$.

Suppose $C_{\eta}$ intersects $B_t(x)$. Since $R=10^5$, 
by Lemma \ref{angle-volume}, the volume of this set is bounded by $k_{\ref{angle-volume}}\delta^2$, where $\delta=c\cdot \frac{t}{\hbox{dist}(x_1,x)}$ and $k_{\ref{angle-volume}}$ is the constant appeared in Lemma \ref{angle-volume}.

Using polar coordinates,  the volume of this part of \eqref{volume of pairs'}
is bounded from above by 
$$k_3 t^2 \int_{z\in B_s(x)}
\frac{1}{\hbox{dist}(z, x)^2} dz \le k_2 s^2t^2.$$

III. The last case is when  $x_1$ is in $B_{Rt}(x)$, 
this element itself would have the freedom of Vol$(B_{Rt}(x))$ which is $O(t^4)$. In this case, choices of  $x_1$ and $x_2$ would multiply to contribute as the rate of $O(t^4)$.  

Summing  the three cases, we finish the proof.
\end{proof}

\subsubsection{Lower bound}

\begin{prop} \label{positive'}
The current $\Phi_H$  satisfies the lower bound in \eqref{leqgeq}.
\end{prop}
\begin{proof} 

Let $s$ be as in Lemma \ref{4.14}, and  $t<10^{-5}  s$.
Fix $x\in M$ and adapted coordinates $(z, w)$ centered at $x$ with radius $s$. 
As in \cite{T1}, we only need to prove  $\Phi_H(i\chi_tdz\wedge d\bar z)>k^{-1} t^4$.

The main picture to have is  Lemma \ref{curves in a ball}, applied to $B_s(x)$. Namely, inside $B_s(x)$, the curves behave as straight lines with respect to the adapted coordinates. 

Fix $\epsilon>0$. 
Recall that $\mathcal M^x$ is a pencil, and identified with $\mathbb C \mathbb P^1$ via $\tau_x$.
Let us begin with choosing a disk $\mathcal C_x\subset \mathcal M^x$  corresponding to a disk  centered at $(1,0)\in \mathbb C \mathbb P^1$.
The latter disk is chosen so that $|dz(\cdot)|\geq 2\epsilon$.  
By Lemma \ref{curves in a ball}, by shrinking $s$ if necessary,  the restriction of $|dz|$ to $C\cap B_s(x)$ is greater than $\epsilon$ for $C\in \mathcal C_x$.
Such an $s$ can be chosen to be independent of $x$.

 As we are estimating the lower bound, so we restrict our attention to $\eta=(x_1, x_2)$ such that  
 $x_1$ is away from $B_{\frac{s}{2}}(x)$ and inside $B_s(x)$, 
 and the line $C_{x, x_1}$ determined by 
$x$ and $x_1$ is contained in the disk $\mathcal C_x\subset \mathcal M$ specified above. 
By Lemma \ref{curves in a ball},  the choices of $x_1$ constitute a compact set with  volume $c_{\epsilon}s^4$.

Now,  fix such an $x_1$.  Consider the set of $x_2$ in $B_{\frac{s}{4}}(x)$, 
for which $\eta=(x_1, x_2)$ contributes to $\Phi_H(if_{B_t(x)}dz\wedge d\bar z)$, namely,
 $C_{x_1, x_2}$ intersects $B_t(x)$. Since we are estimating lower bound, 
we  apply Lemma \ref{curves in a ball} to  count the ones intersecting $B_{\frac{t}{2}}(x)$.
  Since dist$(x_1, x)\geq \frac{s}{2}$,  
 $C_{x_1, x_2}$ intersects
 $B_{\frac{t}{2}}(x)$ as long as  $x_2\in B_{\frac{s}{4}}(x)$ and 
 dist$(x_2, C_{x_1, x})\le  \frac{1}{4} \frac{t}{2}$. 
Thus the volume of $x_2$ is bounded by the radius $ \frac{t}{8}$ tube around  $C_{x_1, x}\cap B_{\frac{s}{4}}(x)$.
By Lemma \ref{2.2}, or Lemma \ref{curves in a ball}, the area of $C_{x_1, x}\cap B_{\frac{s}{4}}(x)$ is bounded from below by $a_x(\frac{s}{4})s^2$. 
Hence the volume of $x_2$ is bounded by $k_1s^2t^2$. 
Notice that, again by Lemma \ref{curves in a ball},   when $t\le 10^{-10}s$,  the rational curve $C_{x_1, x_2}=\pi^{-1}_2(\eta)$  
has the property that the restriction of $|dz|$ to $C\cap B_s(x)$ is greater than  $\frac{\epsilon}{2}$. 
Here $C$ denotes the line $C_{x_1, x_2}$.

By virtue of our choices, $\int_C(if_{B_t(x)}dz\wedge d\bar z)$ is bounded below by $(\frac{\epsilon}{2})^2\cdot a_x(t)$. 
Thanks to Lemma \ref{2.2}, $a_x(t)$ is bounded below as $k_{\ref{2.2}}^{-1}t^2$ (here $k_{\ref{2.2}}$ is the constant appeared in Lemma \ref{2.2}). Given the aforementioned lower bound for $|dz|$ , and given what is said in Lemma \ref{2.2}, it follows that the integral of $if_{B_t(x)}dz\wedge d\bar z$ over $C$ must  be greater than $k_2^{-1}t^2$ by choosing $t(<<s)$ small enough. More precisely, this $k_2^{-1}$ could be chosen as $(\frac{\epsilon}{2})^2 k_{\ref{2.2}}^{-1}$.

Multiplying these three factors:  $c=c_{\epsilon}s^4$, $k_1s^2t^2$ and $k_2^{-1}t^2$ together, we get 
$k^{-1}t^4<\Phi_H(f_{B_t(x)}i\sigma\wedge\bar \sigma)$ for some $k$. 
The constant $k$ is further independent of $x$ and $t$:
\begin{itemize}
\item The first constant $c_{\epsilon}$ is a universal constant since $M$ is compact;
\item The second constant $k_1$, as in the argument, depends on the choices of $x_1$ (and $x$). But our $x_1$ is chosen from a compact set, so $k_1$ is universal as well;
\item The last constant $k_2$ depends on $x_2$ (and $x_1$, $x$), and the tubular neighborhood of $C_{x, x_1}$ we have chosen is compact, so $k_2$ is universal.
\end{itemize}
\end{proof}

 We note that the following theorem due to Gromov \cite{Gr} and Taubes \cite{T} is  an immediate consequence of Propositions \ref{upper'}, \ref{positive'} and \ref{C-F}. 
\begin{theorem}
Any tamed almost complex structure on $\mathbb C \mathbb P^2$ is almost K\"ahler. 
\end{theorem}


\section{Pencils in a big   $J-$nef class}

Let $M$ be a rational manifold and $e$ a class in $S_{K_J}^{+}$. By Proposition \ref{existence'}, $\mathcal M_e$ is non-empty. 
But if $e$ is not 
the  line class on the projective space, then  $\mathcal M_e$ always contains reducible subvarieties, and in fact,  it could entirely consist of
 reducible varieties. 
To guarantee that there are smooth rational curves, we need to restrict to $J-$nef classes in $S_{K_J}^{+}$
(see Remark \ref{no smooth curves}).
\subsection{$J-$nef classes}

Let $J$ be a fixed almost complex structure on $M$.

\begin{definition} \label{nef}
A  homology class $e\in H_2(M; \mathbb Z)$ is said to be $J-$nef if it pairs non-negatively with   any $J-$effective class.

\end{definition}

\subsubsection{Big $J-$nef and $J-$ample classes}

\begin{definition} \label{locus} A $J-$nef class $e$ is said to be  big if $e\cdot e$ is positive.

The vanishing locus $Z(e)$ of a big $J-$nef class $e$  is the union of irreducible subvarieties $D_i$ such that $e\cdot e_{D_i}=0$. Denote  the complement of the vanishing locus of $e$  by $M(e)$.

A  big $J-$nef class $e$  is said to be $J-$ample if $M(e)=M$.
\end{definition}

The following lemma immediately follows from the positivity of intersections of distinct irreducible subvarieties.

\begin{lemma}\label{irr=nef}
If $e$ is   represented by an irreducible $J-$holomorphic subvariety and $e\cdot e>0$, then $e$ is  a big $J-$nef class.

\end{lemma}

\subsubsection{List of main results}
We now list main properties of $J-$nef classes in $S_{K_J}$.
The first  one is Theorem \ref {emb-comp}, proved in \cite{LZrc}. As mentioned in the introduction, for rational manifolds, it plays a fundamental role  to remove a number  of the genericity assumptions in \cite{T1}.

Together with Lemma \ref{pdmanifold}, we have the following important  consequence, also established in   \cite{LZrc}.
\begin{prop}\label{reducible-dim}
Suppose $e\in S_{K_J}$ is a $J-$nef class.
If $\Theta=\{(C_i,m_i), 1\leq i\leq n\}\in \mathcal M_{red, e}$ is connected, then
\begin{equation}\label{red-dim'}
\sum _{(C_i, m_i)\in \Theta} m_i l_{e_{C_i}}\leq l-1.
\end{equation}
In particular, 
\begin{equation}\label{red-dim}
\sum _{(C_i, m_i)\in \Theta}  l_{e_{C_i}}\leq l-1.
\end{equation}

\end{prop}
This  is an analogue of Proposition 3.4 in \cite{T1},
but valid for an arbitrary tamed almost complex structure.
It follows that, similar to Proposition 1.1 in \cite{T1}, $J-$nef classes  have the following property:

\begin{prop} \label{1.1} Suppose $e$ is a  $J-$nef class in $S_{K_J}$ with $e\cdot e\geq 0$.
The map $\pi_l: \mathcal M_{irr, e, l}\rightarrow M^{[l]} $ is onto the complement
of a compact, measure zero subset. In particular, $e$ is represented by  a smooth rational curve. 

\end{prop}

Fix an orthonormal frame for $T_{1,0}M|_x$ to identify the space of complex $1-$dimensional subspaces with $\mathbb C \mathbb P^1$. Consider the map $$\tau^{x, \Omega}:\mathcal M_e^{x, \Omega}\to  \mathbb C \mathbb P^1,        \quad  C\mapsto T_xC.$$ The next result is very useful to understand this map.

\begin{prop}\label{H}
For any  $J-$nef class $e$ in $S_{K_J}^{+}$,
we can choose  a $J-$nef class  $H_e$ in  $S_{K_J}^{\geq 0} $ such that 
$H_e\cdot e=1$ or $2$, and $H_e\cdot e=2$ only if $H_e$ is proportional to $e$.  
\end{prop}

For any $\Theta=\{(C_1, m_1),\cdots, (C_n, m_n)\}\in \mathcal M^x_{red, e}$ with $x\in C_1$,
given  $\Omega\in M^{[l-2]}$, let 
$w_i$ be  the cardinality of $\Omega_i=\Omega\cap C_i$. 

\begin{definition}\label{pretty generic}
Fix a point $x\in M(e)$.   $\Omega\in M^{[l-2]}$ is called {\it pretty generic} with respect to $e$ and $x$ if
\begin{itemize}
\item $x$ is distinct from any entry of $\Omega$;
\end{itemize}
For each $\Theta=\{(C_1, m_1),\cdots, (C_n, m_n)\}\in \mathcal M^x_{red, e}$ with $x\in C_1$;
\begin{itemize} 

\item $x$ is not in $C_i$ for any $i\geq 2$; 

\item  $\Omega_i\cap \Omega_j=\emptyset$;

\item  
$1+w_1\ge m_1e\cdot e_1(\ge l_{e_1})$, and  $w_i\ge m_ie\cdot e_i(\ge l_{e_i})$ for $i\geq 2$. 
\end{itemize}
Let $G_{e}^x$ be the set of pretty generic $l-2$ tuples with respect to $e$ and $x$. 
\end{definition}

\begin{prop}\label{complement}
Suppose $e$ is a big $J-$nef class in $S_{K_J}$ and $x\in M(e)$.
Then the complement of  $G_{e}^x$ has complex codimension at least one in $M^{[l-2]}$.
\end{prop}

Further, big $J-$nef classes have the following property.
We would like to show that if we fix $x$ and choose $\Omega$ pretty generic,
we always have smooth rational curves passing through $x$ and $\Omega$. Moreover, a generic complex direction in $T_{x}M$ would be tangent to these curves at $x$.

\begin{prop}\label{gendirect} Suppose $e$ is big $J-$nef. For $x\in M(e)$ and $\Omega\in G_e^x$, 
The map $\tau^{x, \Omega}$ is well defined and in fact  it is a homeomorphism.

Moreover, $\tau^{x, \Omega}$ is a diffeomorphism away from the reducible curves. 
\end{prop}

Hence, for every complex direction in $T_xM=\mathbb C^2$, 
there is a (possibly reducible) rational curve tangent to it and passing through $\Omega$ and $x$. 
Moreover, except only finitely many directions, this rational curve is smooth.

Propositions \ref{1.1},  \ref{H},  \ref{complement},  \ref{gendirect} will be  proved in the next four subsections respectively.


\subsection{Existence of smooth curves}

In this subsection we prove Proposition \ref{1.1}. For this purpose, we need to estimate the 
dimension of the space of reducible curves.

\subsubsection{Dimension of the moduli space of reducible curves}
For  $$\Theta=\{(C_1, m_1),\cdots, (C_n, m_n)\}\in \mathcal M_{e},$$
let  $\Xi_{\Theta}=(e_{C_1}, \cdots, e_{C_n})$. 
Set  $\chi_e=\{\Xi_{\Theta}|\Theta\in \mathcal M_{red, e}
\}.$
Proposition \ref{compactness} guarantees that $\chi_e$ is a finite set.
Given $\Xi\in \chi_e$, 
the subspace of  reducible curves $\mathcal M_{\Xi}\subset \mathcal M_{e, red}$  corresponding to  $\Xi$
is naturally identified with  $\times_{e'\in \Xi} \mathcal M_{e', irr}$.

\begin{prop}\label{reducible-dim'}
Suppose $e$ is a  $J-$nef class in $S_{K_J}$. Then
$\mathcal M_{red,e}=\cup_{\Xi\in \chi_e}\mathcal M_{\Xi}$ is a finite union of manifolds, each with  complex dimension  at most $l-1$.
\end{prop}

\begin{proof}
By Lemma \ref{pdmanifold} and Theorem \ref{emb-comp}, for each $\Xi\in \chi_e$, 
$\mathcal M_{\Xi}$
 is a manifold of complex dimension $\sum _{(C, m)\in \Theta} l_{e_C}$.
Thus the  assertion follows from Proposition \ref{reducible-dim}.
\end{proof}

\begin{remark}\label{top stratum} In \cite{LZrc} we  completely determine the possible configuration with $l= 1+\sum_{i=1}^n l_{e_i}$.

$l= 1+\sum_{i=1}^n l_{e_i}$  if and only if  each multiplicity is  $1$, and

$\Theta$ is one of the following configurations: 

$\bullet$ If $\Theta_-$ is empty then 
$n=2$,   $e_1\cdot e_2=1$,  $e_i\cdot e_i\geq 0$.

$\bullet$ If $\Theta_{-}$ is not empty and there is no $-1$ curves, then 
$\Theta_{-}$ consists of a unique element $(C_1, 1)$ with $e_1\cdot e_1=1-n\leq -2$, and 
$\Theta_{+}$  consists of  at least $n-1\geq 2$ elements, $e_i=\cdots= e_{n}$ and $e_i\cdot e_i=0$ for $i\geq 2$.
Moreover, $e_1\cdot e_2=1$. In short,   
it is a comb like configuration.  

$\bullet$ If $\Theta_{-}$ contains a $-1$ curve, then  either  $\Theta$ is a successive infinitely near  blow-up of
a  smooth rational curve with non-negative self-intersection,  
Equivalently,  it means that, starting from the second blow-up, we only blow up at a point in a component with negative self-intersection.

$\bullet$ or a successive infinitely near  blow-up
of a comb like configuration at points  in $C_1$.  
Here infinitely near blow up means that all the blow ups, from the second one on, occur at some  point 
not lying in the proper transform of the original configuration. Equivalently,  it means that  we successively blow up in the union of  components of negative self-intersection.

\end{remark}

\subsubsection{Proof of Proposition \ref{1.1}}.

\begin{proof} 
It follows from Proposition \ref{existence} 
that $\pi_{irr, l}$ is surjective. On the other hand, 
by Proposition \ref{reducible-dim'}, the image of $\pi_{red, l}$ is of codimension at least 2. 
The  assertion follows. 
\end{proof}


\subsection{Intersection properties}
In this subsection we establish Proposition \ref{H}.

\subsubsection{Intersection properties of $K_J-$spherical classes}
\begin{lemma}\label{pairS+}The following intersection properties hold:
\begin{itemize}

\item $S_{K_J}^{+}$ pairs positively with $S_{K_J}^{\geq 0}$.

\item $S_{K_J}^{+}$ pairs non-negatively with $\mathcal E_{K_J}$.

\item $\mathcal E_{K_J}$ pairs non-negatively with $\mathcal E_{K_J}$.
\end{itemize}
\end{lemma}
\begin{proof} Since $b^+=1$, and $S_{K_J}^{\geq 0}$ pairs positively with any $J-$tamed symplectic form $\omega$, $S_{K_J}^{+}$ pairs positively with $S_{K_J}^{\geq 0}$ by light cone lemma.
 
Any element in $\mathcal E_{K_J}$ is $J'-$effective for any tamed $J'$ with $K_{J'}=K_J$ by Proposition \ref{existence'}. For a generic tamed $J'$, $e\in S_{K_J}^{+}$ could be represented by smooth irreducible $J-$holomorphic curves. The second statement then follows from the positivity of intersections of distinct irreducible subvarieties.

The last item follows because any two elements $e_1$ and $e_2$ in $\mathcal E_{K_J}$ have irreducible representations for a generic tamed $J'$ with $K_{J'}=K_J$.
\end{proof}

\begin{lemma}\label{H''}

For any  $e$ in $S_{K_J}^{+}$,
we can choose  a class  $H_e$ in  $S_{K_J}^{\geq 0} $ such that 
 $H_e\cdot e=1$ or $2$, and   $H_e\cdot e=2$ only if $H_e$ is proportional to $e$.  

\end{lemma}

\begin{proof}
  A rational manifold is either $\mathbb C \mathbb P^2\#k\overline{\mathbb C \mathbb P^2}$
or $S^2\times S^2$.

$\bullet$ $S^2\times S^2$

Denote the two factors classes by $H_1$ and $H_2$ such that $K_J=-2H_1-2H_2$. 
For the class $aH_1+bH_2$ of an irreducible curve, the adjunction formula implies
\begin{equation} \label {ad1}\begin{array} {ll} &(aH_1+bH_2)^2+(-2H_1-2H_2)(aH_1+bH_2)+2\cr
=&2ab-2b-2a+2=2(1-a)(1-b)\geq 0.\cr
\end{array}
\end{equation} is clear that   $S_{K_J}$ is contained in the following two sequences of  classes:
$$A_l=H_1+lH_2, \quad  B_l=lH_1+H_2, \quad l\in \mathbb Z.$$
The sequence $A_l$ is in $S_{K_J}^{+}$
if $l>0$, and $A_l\cdot A_l=0$ if $l=0$. The same is true for the sequence $B_l$.

For $e=A_l$ with $l>0$, choose $H_e=B_0=H_2$,
and for $e=B_l$ with $l>0$ choose $H_e=A_0=H_1$.

For $\mathbb C \mathbb P^2\#k\overline{\mathbb C \mathbb P^2}$ there exists a basis of spherical classes $H, E_1, ..., E_k$ with $H\cdot H=1, E_i\cdot E_i=-1$ such that 
$K_J=-3H+E_1+\cdots +E_k$. 

$\bullet$  $\mathbb C \mathbb P^2$

  $H$ and $2H$  are the only classes in $S_{K_J}$.

$\bullet$    $\mathbb C \mathbb P^2\#\overline{\mathbb C \mathbb P^2}$

For $\mathbb C \mathbb P^2 \# \overline{\mathbb C \mathbb P^2}$ and class $\alpha H+ \beta E$, the adjunction formula is of the form
\begin{equation} \label {ad2}\begin{array} {ll}
(\alpha-1)(\alpha-2)- \beta(\beta+1)
\geq 0\cr
\end{array}
\end{equation}
 $S_{K_J}$ is a subset of the following sequence $$D_s= sH+(1-s)E, \quad s\in \mathbb Z.$$

Choose  $H_e=H-E$.

$\bullet$    $\mathbb C \mathbb P^2\#k\overline{\mathbb C \mathbb P^2}$ with $k\geq 2$

When $k\geq 2$, it is hard to explicitly describe the classes in $S_{K_J}^{\geq 0}$.
We invoke the classification  up to Cremona equivalence (see \cite{LBL}). 
As noted in \cite{LBL}, 
any  class is in $S_{K_J}^{\geq 0}$ is Cremona equivalent to one of the following classes
\begin{enumerate}
 \item $2H$, $H$, 
 \item $(n+1)H-nE_1, n\ge 1$,
 \item $(n+1)H-nE_1-E_2, n\ge 1$.
\end{enumerate}

Case (1).  If $e$ is equivalent to $2H$ or $H$, we choose $H_e=\frac{e}{\sqrt{e\cdot e}}$. The class $H_e$ is $J-$nef since it is proportional to the $J-$nef class $e$. 

Case (2). If $e$ is equivalent to $(n+1)H-nE_1, n\geq 1$, we choose $H_e$ to be $H-E_1$ under the same equivalence. Notice that  $H-E_1\in S_{K_J}^0$ and $H_e\cdot e=1$.

Case (3).  When the class $e$ is equivalent to $2H-E_1-E_2$,  then  we could also choose $H_e=e$.

Case (4) If $e$ is equivalent to $(n+1)H-nE_1-E_2, n\geq 2$, we again choose $H_e$ to be $H-E_1$ under the same equivalence. Notice that  $H_e\cdot e=1$.
\end{proof}

\subsubsection{$J-$nef classes on $S^2-$bundles over $S^2$}
For $S^2\times S^2$, 
the negative self-intersection classes must be of the form $aH_1+bH_2$ with $ab<0$.

It follows from (\ref{ad1})
that class of any irreducible curve satisfies $(1-a)(1-b)\geq 0$.
Thus the  only possible negative square irreducible $J-$curves are in the classes $A_p$ with $p<0$ or $B_p$ with $p<0$. 

Moreover, given $J$,  there is at most one such curve by the positivity of intersections.

Case (i). There are irreducible $J-$curves with  negative self-intersection. 

Case (ii).  $A_p$ is $J-$effective for some $p< 0$.

Case (iii).   $B_p$ is $J-$effective for some  $p< 0$.

The negative self-intersection  classes must be of the
 form $aH+bE$ with $|a|< |b|$.
Then the only possible negative square irreducible $J-$curves are in the classes
$$ (1-s)H+sE_1, \quad s>0.$$
These classes are in $S_{K_J}$. 
Moreover, there is at most one such curve due to positivity of intersections.

In summary, 

\begin{lemma} \label{s2bundle-nef}
For $S^2\times S^2$, 

In case (i),  both  $A_l$ and  $B_l$ are $J-$nef if $l\geq 0$, and $J-$ample if $l>0$. 

In case (ii),   $A_l$  is $J-$nef when $l\geq -p$, and $J-$ample if $l\geq -p+1$. $B_0$ is $J-$nef.

In case (iii),  $B_l$    is $J-$nef  when $l\geq -p$, and ample if $l\geq -p+1$. $A_0$ is $J-$nef.

For  $\mathbb C \mathbb P^2\#\overline{\mathbb C \mathbb P^2}$,
$H-E$ is $J-$nef for any $J$. If $ sH+(1-s)E_1$ is $J-$effective for some $s\leq 0$, then $D_l$ is $J-$nef for
$l\geq 1-s$, and $J-$ample for $l\geq 2-s$.
\end{lemma}

Notice that for an  $S^2-$bundle over $S^2$, there is always a $J-$ample class,
and there is always a $J-$nef class with self-intersection $0$.

\subsubsection{A criterion for $H-E_i$ to be $J-$nef}
For $\mathbb C \mathbb P^2\#k\overline{\mathbb C \mathbb P^2}$ 
we need the following observation.

\begin{lemma} Suppose there is an irreducible curve class pairing negatively with $H-E_1$, say $e_C=aH-b_1E_1-...-b_nE_n$. Then $a\leq 0$.  
The same is true for $H-E_i$. 
\end{lemma}

\begin{proof} To see this, observe that $(H-E_1)\cdot e_C<0$ means 
$a<b_1$.  If $a>0$ then $b_1\geq 2$. 
Now the $K_J-$adjunction number 
$$e_C\cdot e_C+K_J\cdot e_C\le a^2-b_1^2-3a+b_1\le (b_1-1)^2-3(b_1-1)-b_1^2+b_1=-4b_1+4$$
is less than $-2$, which is impossible. 
\end{proof}

\subsubsection{Proof of Proposition \ref{H}}

\begin{proof}
 Suppose $e$ is $J-$nef. 
It suffices to show that we can further choose $H_e$ in Lemma \ref{H''} to be $J-$nef. 

For $S^2\times S^2$, 
if $e=A_l$ with $l>0$, $H_e=B_0=H_2$ is $J-$nef. 
The second case is similar. 

For  $\mathbb C \mathbb P^2$,  $H$ and $2H$  are $J-$nef for any tamed $J$. 
 In both cases, we choose $H_e$ to be $H$. 

For $\mathbb C \mathbb P^2\# k\overline{\mathbb C \mathbb P^2}$, cases (1) and (3) are clear. For cases (2) and (4) we will show that $H_e=H-E_1$ is $J-$nef.

Suppose there is an irreducible curve class pairing negatively with $H-E_1$, say $e_C=aH-b_1E_1-...-b_nE_n$. 
By the lemma above, $a\leq 0$. 

If   $e=(n+1)H-nE_1, n\geq1,$ is $J-$nef, then $H\cdot e_C\ge -n(H-E_1)\cdot e_C>0$. This implies that $a>0$.

If  $e=(n+1)H-nE_1-E_2, n\ge 2$ is $J-$nef, then  $$(H-E_2)\cdot e_C\ge - n(H-E_1)\cdot e_C\ge n.$$ This means $a\ge b_2+n$. Since $a\le 0$, we have $b_2<a\le 0$. 
Thus the $K_J-$adjunction number $$e_C\cdot e_C+K_J\cdot e_C\leq a^2-b_2^2-3a+b_2\le -2n\leq -4,$$ which is  impossible. 
\end{proof}


\subsection{Vanishing locus and reducible curves}

Suppose $e$ is a big $J-$nef class  in $S_{K_J}$, especially $l=e^2+1\geq 2$. 
Recall the complement of the vanishing locus of $e$ is denoted by $M(e)$. Fix $x\in M(e)$. 

\subsubsection{Reducible rational curves  through a point in a big $J-$nef class}

Consider $\Theta=\{(C_1, m_1),\cdots, (C_n, m_n)\}\in \mathcal M^x_{red, e}$
and assume $x\in C_1$. 
Denote $e_{C_i}$ by $e_i$.

Set  $\chi^x_e=\{\Xi_{\Theta}|\Theta\in \mathcal M^x_{red, e}\}.$
Given $\Xi\in \chi^x_e$, 
the subspace of  reducible curves $\mathcal M^x_{\Xi}\subset \mathcal M^x_{red, e}$  corresponding to  $\Xi$
is naturally identified with $\mathcal M^x_{irr, e_1} \times \times_{i\ge 2}  \mathcal M_{irr, e_i}$.

We  estimate the dimension of $\mathcal M^x_{red, e}$.

\begin{prop}\label{diresphere}
Suppose $e$ is a big $J-$nef class  in $S_{K_J}$. For any $x\in M(e)$,
$\mathcal M_{red,e}^x$ is a union of manifolds with
 complex dimension at most $l-2$.

\end{prop}
\begin{proof} Given $\Theta=\{(C_1, m_1),\cdots, (C_n, m_n)\}\in \mathcal M^x_{red, e}$, let $\Xi=\Xi_{\Theta}$. 

By Lemma \ref{pdmanifold} and Theorem \ref{emb-comp},  $\mathcal M_{\Xi}^x$
has complex  dimension 

$$\dim_{\mathbb C}  \mathcal M^x_{irr, e_1}+  \sum_{i=2}^n l_{e_i}.$$

First  suppose that  $e_{1}\cdot e_{1}\geq 0$. By Lemma \ref{pdmanifold},
$\mathcal M_{irr, e_{1}}^x$ has complex dimension $l_{e_1}-1$.
Thus  we need to show $$l_{e_1}-1+\sum_{i= 2}^n l_{e_i}\leq l-2,$$
 which is the same as \eqref{red-dim}.

Now suppose that 
$e_1\cdot e_1<0$. Then we need to show $\sum_{i=2}^n l_{e_i}\leq l-2$.

We assume $l_{e_j}=0$ when $j\le k$, and $l_{e_j}> 0$ when $j\ge k+1$.

By Lemma 2.7 in \cite{LZrc},  $\Theta$ is connected, so $e_j\cdot (e-m_je_j)\geq 1$ for each $j\geq 2$.
Therefore $l$ can be estimated as follows:
\begin{equation} \label{3terms}
\begin{array}{lll}
 l-1&=&e\cdot e\\
 &&\\
 &=&(\sum_{j=2}^n m_j {e_j}+m_1e_1)\cdot e\\
 &&\\
 &\ge&\sum_{j=k+1}^n ( m_j {e_j} \cdot ( m_je_j+(e-m_je_j))+      m_1e_1\cdot e\\
 &&\\
  &=&\sum_{j=k+1}^n(m_j^2e_j\cdot e_j+ m_je_j\cdot (e-m_je_j))  +  m_1e_1\cdot e\\
 &&\\
 &\geq&\sum_{j=k+1}^n m_jl_{e_j}+  m_1e_1\cdot e\\
 &&\\
 &=& \sum_{j=1}^n m_jl_{e_j} +  m_1e_1\cdot e.\\
 \end{array}
 \end{equation}

Recall that we assume $x\in C_1$ and  
$x\in M(e)$,  therefore $e\cdot e_1>0$. 
Hence in this case, we have also shown that   $\mathcal M^x_{\Xi}$ has complex dimension 
at most $l-2$.

\end{proof}

This can be viewed as a version of Lemma 4.7 in \cite{T1}.

Now consider the subset $\mathcal M_{red,e}^x(x\hbox{ nodal})$ of
$\mathcal M_{red,e}$ where $x\in C_1$ and $x\in C_i$ for some $i\geq 2$. 

\begin{lemma}\label{nodal}
Suppose $e$ is a big $J-$nef class  in $S_{K_J}$. For any $x\in M(e)$,
$\mathcal M_{red,e}^x(x\hbox{ nodal})$ is a union of manifolds with
 complex dimension at most $l-3$.

\end{lemma}
\begin{proof}
The proof is similar to that of Proposition \ref{diresphere}. 

Given $\Theta=\{(C_1, m_1),\cdots, (C_n, m_n)\}\in \mathcal M^x_{red, e}(x\hbox{ nodal})$, let $\Xi=\Xi_{\Theta}$. Assume without loss of generality that $x\in C_2$. 

By Lemma \ref{pdmanifold} and Theorem \ref{emb-comp},  $\mathcal M_{\Xi}^x(x\in C_2)$
has complex  dimension 

$$\dim_{\mathbb C}  \mathcal M^x_{irr, e_1}+ \dim_{\mathbb C}  \mathcal M^x_{irr, e_2}+  \sum_{i=3}^n l_{e_i}.$$

We need to divide into four cases: I. $e_1\cdot e_1\geq 0, e_2\cdot e_2\geq 0$.
II. $e_1\cdot e_1<0, e_2\cdot e_2<0$.
III. $e_1\cdot e_1\geq 0, e_2\cdot e_2<0$.
IV. $e_1\cdot e_1<0, e_2\cdot e_2\geq 0$.

In Case I,  $$\dim \mathcal M^x_{\Xi}(x\in C_2)= l_{e_1}-1+l_{e_2}-1+\sum_{i= 3}^n l_{e_i}\leq l-3$$ by Lemma \ref{pdmanifold} and 
 \eqref{red-dim}. 
 
 In Case II, we also assume $l_{e_j}=0$ when $j\le k$, and $l_{e_j}> 0$ when $j\ge k+1$. By Lemma 2.7 in \cite{LZrc},  $\Theta$ is connected, so $e_j\cdot (e-m_je_j)\geq 1$ for each $j\geq 3$.
Therefore $l$ can be estimated as follows:
\begin{equation} \nonumber
\begin{array}{lll}
 l-1&=&e\cdot e\\
 &&\\
 &=&(\sum_{j=3}^n m_j {e_j}+m_1e_1+m_2e_2)\cdot e\\
 &&\\
 &\ge&\sum_{j=k+1}^n ( m_j {e_j} \cdot ( m_je_j+(e-m_je_j))+      m_1e_1\cdot e + m_2e_2\cdot e\\
 &&\\
  &=&\sum_{j=k+1}^n(m_j^2e_j\cdot e_j+ m_je_j\cdot (e-m_je_j))  +  m_1e_1\cdot e + m_2e_2\cdot e\\
 &&\\
 &\geq&\sum_{j=k+1}^n m_jl_{e_j}+  m_1e_1\cdot e + m_2e_2\cdot e\\
 &&\\
 &=& \sum_{j=1}^n m_jl_{e_j} +  m_1e_1\cdot e+ m_2e_2\cdot e.\\
 \end{array}
 \end{equation}

Recall that we assume $x\in C_1\cap C_2$ and  
$x\in M(e)$,  therefore $e\cdot e_1>0$, $e\cdot e_2>0$. 
Hence in this case, we have also shown that   $\mathcal M^x_{\Xi}$ has complex dimension 
at most $l-3$. 

Cases III and IV are similar, we only prove Case IV. In this case, we need to show $\sum_{i=2}^n l_{e_i}\leq l-2$, which is exactly the second case of Proposition \ref{diresphere}. 
\end{proof}
\subsubsection{Pretty generic $l-2$ tuple}

\begin{lemma} \label{finite} For any  $x\in M(e)$ and  $\Omega\in G_e^x$,
any two elements in  $\mathcal M_e^{x, \Omega}$
 intersect only at the $l-1$ points 
$x, \Omega$. So it is like a pencil in algebraic geometry. Moreover, 
\begin{itemize}

\item Given $z$ distinct from $x$ and $\Omega$, $\mathcal M^{z, x, \Omega}$ consists of unique curve.

\item There is a unique curve in $e$ passing through 
the $l-1$ points $x, \Omega$ and a given direction at one of these points.

\item $\mathcal M^{x, \Omega}_{red}$ is a finite set.
Moreover,  these reducible curves cannot be tangent to each others at $x$.

\item We can define the map
$\tau^{x, \Omega}:\mathcal M^{x, \Omega}\to  \mathbb C \mathbb P^1,        \quad  C\mapsto T_xC.$

\end{itemize}
\end{lemma}
\begin{proof} 
Suppose $\Theta, \Theta'$ are two elements in $\mathcal M^{z, x, \Omega}$. If $\Theta, \Theta'$ have no common components, then the the first bullet follows from local positivity of intersection. Hence we assume they have common components.

We rewrite two subvarieties $\Theta, \Theta' \in \mathcal M_e$, allowing $m_i=0$ in the notation, such that they have the same set of irreducible components formally, i.e. $\Theta=\{(C_i, m_i)\}$ and $\Theta'=\{(C_i, m'_i)\}$. Then for each $C_i$, if $m_i\le m'_i$, we change the components to $(C_i, 0)$ and $(C_i, m'_i-m_i)$. Apply this process to all $i$ and discard finally all components with multiplicity $0$ and denote them by $\Theta_0,\Theta'_0$ and still use $(C_i, m_i)$ and $(C_i, m'_i)$ to denote their components. Notice they are homologous, formally have homology class $e-\sum m_{k_i}[C_{k_i}]-\sum m'_{l_j}[C_{l_j}]$ with pairwise different $C_{k_i}$ and $C_{l_j}$.
For each $C_{k_i}$, we know that there should be originally at least $m'_{k_i}e\cdot e_{k_i}$ points on it.  Notice $\sum m_i e\cdot e_i=\sum m'_i e\cdot e_i=e^2=l-1$, and none of $x, \Omega$ are nodal points of $\Theta$ and $\Theta'$. Hence either $e\cdot e_{k_i}=0$, or $e\cdot e_{k_i}=1$ and $m'_{k_i}-m_{k_i}=1$. And the latter case could  happen at most once. Similarly for $C'_{l_j}$. Therefore we know there are at least $e\cdot e_{\Theta_0}$ intersections (among $x, \Omega$) of $\Theta_0, \Theta'_0$.

Notice  $e\cdot e_{\Theta_0}> e_{\Theta_0}^2$ since original $\Theta, \Theta'$ have at least one common component. Hence there are more intersections  than the homology intersection number $e_{\Theta_0}^2$ of our new subvariety $\Theta_0$ and $\Theta_0'$. Then our claim follows from the local positivity of intersection.

The second bullet follows by the same argument.

For the third bullet, first notice  $\chi^{x, \Omega}_{red, e}$ is a finite set by Proposition \ref{compactness}. Hence we could fix the type  $\Xi=(e_i)\in \chi^{x, \Omega}_{red, e}$. Apply the same process, the removed components should have the same multiplicities since $l-1=\sum m_ie\cdot e_{C_i}$.  And if $\Theta\neq \Theta_0$, the above argument claims contradiction. Hence, it implies there are no common components of $\Theta, \Theta'$. Especially, it implies there are no spheres of negative self-intersection as components. Hence, it is the situation that Corollary 2 in \cite{Sik} applies, which says that $\mathcal M_{red}^{x, \Omega}$ is isolated in the compact space $\mathcal M^{x, \Omega}$.
Hence $\mathcal M_{red}^{x, \Omega}$ is a finite set.

Since $x$ is not a nodal point of any $\Theta$, and thanks to the second bullet, the map $\tau^{x, \Omega}$ is well-defined.
\end{proof}

\subsubsection{Proof of Proposition \ref{complement}}

\begin{proof}
The complement of $G_e^x$ is the union of the  four sets $V_i, i=1,2,3,4$:  $V_i$ is the set of points in $M^{[l-2]}$  violating the $i-$th item of Definition \ref{pretty generic}, but not the previous items.  

It is easy to see that $V_1$  has complex codimension $2$.

To estimate the dimensions of $V_2, V_3, V_4$, consider the map 
$$\pi^x_{red,l-2}:\mathcal M^{x}_{ red, e, l-2}\to M^{[l-2]}.$$

We first  deal with $V_2$.  For each $\Xi=(e_i) \in \chi_{e}^x$, let $V_2(\Xi)$ be the image of the map $\pi^x_{red,l-2}$ restricted to  $\mathcal M^x_{\Xi, l-2}(x \hbox{ nodal})$.
Here $\mathcal M^{x}_{\Xi, l-2}(x \hbox{ nodal})\subset \mathcal M^{x}_{\Xi, l-2}$  
consists of 
$\Theta= \{(C_i, m_i)\}\times \Omega\in   \mathcal M^{x}_{\Xi} \times M^{[ l-2]}$  with 
$x\in C_1$,  and   $x\notin \Omega$.

Clearly,  $V_2$ is the union of $V_2(\Xi)$ over $\Xi=(e_i) \in \chi_{e}^x$.
By Lemma \ref{nodal}, $$\dim_{\mathbb C} \mathcal M^x_{\Xi, l-2}(x\hbox{ nodal})\leq (l-3)+(l-2)=2(l-2)-1.$$
 Thus $V_2$ has  complex codimension at least $1$ in $M^{[l-2]}$. 

For the set $V_3$, it is similarly the union of $V_3(\Xi)$ over $\Xi=(e_i) \in \chi_{e}^x$, where $V_3(\Xi)$ is the 
image of  the map $\pi^x_{red,l-2}$ restricted to  $\mathcal M^x_{\Xi, l-2}(\Omega \hbox{ nodal})$.
Here $\mathcal M^{x}_{\Xi, l-2}(\Omega \hbox{ nodal})\subset \mathcal M^{x}_{\Xi, l-2}$  
consists of 
$\Theta=\{(C_i, m_i)\}\times \Omega\in   \mathcal M^{x}_{\Xi} \times M^{[ l-2]}$  with 
$x\in C_1$,  and 

$\bullet$ $x$ is not in $ C_i$ for any $i\geq 2$,

$\bullet$ there exists $i\ne  j$ and   $y\in  M$ such that $y\in \Omega\cap C_i$ and $y\in \Omega\cap C_j$, i.e. $y\in \Omega$ is a nodal point of  $\Theta$.

By  Proposition \ref{diresphere}, $\dim_{\mathbb C} \mathcal M^x_{\Xi}\le l-2$. Given $\{(C_i, m_i)\}\times \Omega\in \mathcal M^x_{\Xi, l-2}(\Omega \hbox{ nodal})$, 
there is a nodal   point  $y$ in $\Omega$. Observe that such a nodal point is in the intersection of $C_i$ and $C_j$, so it  has $0$ dimensional freedom.   
Hence 
$$\dim_{\mathbb C}\mathcal M^x_{\Xi, l-2}(\Omega \hbox{ nodal})\le (l-2)+(l-3)=2(l-2)-1.$$ Thus $V_3$ has  complex codimension at least $1$.

Finally, we deal with $V_4$. For each $\Xi=(e_i) \in \chi_{e}^x$, define
$\mathcal M^{x}_{\Xi, l-2}(\Omega_i \hbox{ non-generic})\subset \mathcal M^{x}_{\Xi, l-2}$ 
consisting of 
$\Theta=\{(C_i, m_i)\}\times \{\Omega_i\}\in   \mathcal M^{x}_{\Xi} \times M^{[ l-2]}$  with 
$x\in C_1$,  and 

$\bullet$ $x$ is not in $ C_i$ for any $i\geq 2$, 

$\bullet$ $\Omega_i\subset C_i$,  

$\bullet$ $\Omega_i\in M^{[|\Omega_i|]}$, $\Omega_i\cap \Omega_j=\emptyset$, and $\sum |\Omega_i|=l-2$,

$\bullet$ $1+|\Omega_1|< l_{e_1}$ or $|\Omega_i |< l_{e_i}$ for $i\geq 2$.

Clearly, under $\pi^x_{red, l-2}$, the union of the  image of $\mathcal M^{x}_{\Xi, l-2}(\Omega_i \hbox{ non-generic})$ over $\Xi=(e_i) \in \chi_{e}^x$  is the rest part of the complement of  $G_{e}^x$.

By the estimate \eqref{3terms}, $\dim_{\mathbb C} \mathcal M^{x}_{\Xi}\le l-2$ if $x\in M(e)$.
If $\dim_{\mathbb C} \mathcal M^{x}_{\Xi}<l-2 $, then already
$$\dim_{\mathbb C} \mathcal M^{x}_{\Xi, l-2} \le (l-3)+(l-2)=2(l-2)-1.$$

So we assume that $\dim_{\mathbb C} \mathcal M^{x}_{\Xi}=l-2 $. In this case, $\sum_i l_{e_i}=l-1$. And for any $\{(C_i, m_i)\}\times \{\Omega_i\}\in \mathcal M^{x}_{\Xi, l-2}(\Omega_i \hbox{  non-generic})$,
either $1+|\Omega_1|< l_{e_1}$, or $|\Omega_j |< l_{e_j}$ for $j\geq 2$.

Observe that $\pi^x_{red,l-2}$ restricted to $ \mathcal M^{x}_{\Xi, l-2}$ is of the form:
$$ \times \pi_{e_i, |\Omega_i|}: \mathcal M_{e_1}^{x, \Omega_1}\times_{i\ge 2} \mathcal M^{\Omega_i}_{e_i}\to \times_i M^{[|\Omega_i|]}. $$
The source $ \mathcal M^{x}_{\Xi, l-2}$ has total complex dimension $2(l-2)$. But when some $|\Omega_i|<l_{e_i}$ (or $1+|\Omega_1|< l_{e_1}$), 
$ \pi_{e_i, |\Omega_i|}$  drops dimension since $$\dim_{\mathbb C}  \mathcal M_{e_i, l_{e_i}-p}= l_{e_i}+ l_{e_i}-p>  \dim_{\mathbb C}  M^{l_{e_i}-p}= 2(l_{e_i}-p).$$

\end{proof}

\subsection{Abundance of pencils} \label{pencilbignef}
In this subsection we establish Proposition \ref{gendirect}.

\subsubsection{$\mathcal M_e^{x, \Omega}$ is homeomorphic to $S^2$}

\begin{prop}\label{H'}
Let  $e$ be a big $J-$nef  class in $S_{K_J}$. 
Fix a point $x\in M(e)$ and choose $\Omega\in  G_e^x$. Then $\mathcal M_e^{x, \Omega}$ is homeomorphic to $S^2$.
\end{prop}
\begin{proof}

By Lemma \ref{H},  there is another $J-$nef  class $H_e$ in $S_{K_J}$  such that $H_e\cdot e=1$ or $2$.
We prove that $\mathcal M_e^{x, \Omega}$ is homeomorphic to a smooth representative of $H_e$. 

Let us first assume that $H_e\cdot e=1$.

By the first item of Proposition  \ref{1.1}, we can choose a smooth  rational curve  $S$ representative of $H_e$ such that it does not pass through any entry of $\Omega$ and $x$. This is possible since $H_e$ is $J-$nef and 
the space of reducible $H_e-$curves is of codimension at least 1 by Proposition \ref{reducible-dim'}. Moreover, 
the space of irreducible $H_e-$curves containing $x$ or any entry of $\Omega$ is of codimension 1 by Proposition \ref{pdmanifold}.

Given any  $z\in S$, $z$ is distinct from $x$ or any entry of $\Omega$. 
By the first bullet of Lemma \ref{finite}, 
there is a unique (although possibly reducible) rational curve $C_{x, z, \Omega}$ in class $e$ passing through $x$, $z$ and $\Omega$.
Thus we obtain a  map $h:z\mapsto C_{x, z, \Omega} $ from $S$ to $\mathcal M_e^{x, \Omega}$. 

The map $h$ is  surjective since $H_e\cdot e\ne 0$. Since $S$ is also $J-$holomorphic and $H_e\cdot e=1$
any curve in $\mathcal M_e^{x, \Omega}$ intersects with $S$ at a unique point by the positivity of intersection. 
Therefore $h$  is also one-to-one.   

Now let us show that $h$ is a homeomorphism, namely both $h$ and $h^{-1}$ are continuous. 
Since $S=S^2$ is Hausdorff  and $\mathcal M_e^{x, \Omega}$ is compact, 
if we can  show that $h^{-1}:\mathcal M_e^{x, \Omega}\to S$ is continuous, it follows that $h$ is also continuous.
 To show $h^{-1}$  is continuous, consider a sequence $C_i \in \mathcal M_e^{x, \Omega}$ approaching to its Gromov-Hausdorff limit $C$. Let the intersection of $C_i$ (resp. $C$) with $S$ be $p_i$ (resp. $p$). Then $p_i$ has to approach $p$ by the first item of the definition of topology on $\mathcal M_e$.
 Therefore
$h$ is a homeomorphism.  

The case that  $H_e\cdot e=2$ is similar. 
We choose the  smooth rational curve $S$ representative of $H_e$ such that it passes $x$ but not any entry of $\Omega$. 
This is achieved by Proposition \ref{diresphere} and  Proposition \ref{pdmanifold} applied to the $J-$nef class $H_e$. 
Here we also need to use the fact that in this case $H_e$ is proportional to $e$ and hence $x$ is also in $M(H_e)$.

Then we vary $z$ in $S$. If $z\neq x$, we choose the rational curve $C_{x, z, \Omega}$ in 
class $e$ passing through $x$, $z$ and $\Omega$. If $z=x$, we choose the rational curve 
$C_{x, x, \Omega}$ in class $e$ passing through $x$, $\Omega$ and tangent to $S$ at $x$. 
The sphere $C_{x, z, \Omega}\in \mathcal M_e^{x, \Omega}$ is unique by the second bullet of Lemma \ref{finite}. 
We thus again obtain a map $h:z\mapsto C_{x, z, \Omega} $  from $S^2$ to $\mathcal M_e^{x, \Omega}$. This map is clearly surjective. Since $S$ is $J-$holomorphic with $x\in S$ and $H_e\cdot e=2$,
any curve in $\mathcal M_e^{x, \Omega}$ either intersects with $S$ at a unique point other than $x$ or is tangent to $S$ at $x$ by the positivity of intersection. 
Therefore $h$  is also one-to-one.   
Now we show that this map is a homeomorphism. As before, we only need to show that $h^{-1}:\mathcal M_e^{x, \Omega}\to S$ is continuous. Again,  consider a sequence $C_i \in \mathcal M_e^{x, \Omega}$ approaching to its Gromov-Hausdorff limit $C$. Let the intersection of $C_i$ (resp. $C$) with $S$ be $p_i$ and $x$ (resp. $p$ and $x$). If $C_i$ (or $C$) tangent to $S$, let $p_i$ (or $p$) be $x$. Then $p_i$ has to approach $p$ by the first item of the definition of topology on $\mathcal M_e$. Therefore $h$ is a homeomorphism.  
\end{proof}

\subsubsection{Proof of Proposition \ref{gendirect}}

\begin{proof} Fix an orthonormal frame for $T_{1,0}M|_x$ to identify the space of complex $1-$dimensional 
subspaces with $\mathbb C \mathbb P^1$. 
Consider the map $$\tau^{x, \Omega}:\mathcal M_e^{x, \Omega}\to  \mathbb C \mathbb P^1,        \quad  C\mapsto T_xC.$$

By Proposition \ref{H'},  $\tau^{x, \Omega}$ 
is a map from $S^2$ to $S^2$.
An injective continuous map from $S^2$ to $S^2$ has to be a homeomorphism.

By the second assertion of Lemma \ref{uniqueness}, these curves cannot be tangent to each other at $x$  if one of them is irreducible. 
Moreover, by the third bullet of   Lemma \ref{finite} there are finitely many reducible curves in this family. In fact, we have shown that $|\mathcal M^{x, \Omega}_{red}|$ is bounded
by $|\chi_{red,e}|$, which only depends on $e$ and $J$. 
We also know that these reducible curves cannot be tangent to each others. To summarize, the map is injective.
\end{proof}


\section{Spherical Taubes currents from big $J-$nef classes}

 \subsection{Weak Taubes currents}

We begin with introducing the notion of spherical current   from a big $J-$nef class.

\subsubsection{Spherical currents}\label{Scurrent}
We continue to assume $J$ is a tamed almost complex structure. Suppose $e$ is a big $J-$nef sphere class  in $S_{K_J}$.  Then $l_e=\iota_e\ge 2$.
Fix $x\in M$. Let us first choose a constant $r_0$. Let $\mathcal A$ be the measure zero set of non-pretty-generic points with respect to $x$ when $x\in M$ if $M(e)=M$ or $x$ belongs to a compact subset $K\subset M(e)$ if $M(e)\subsetneq M$. Choose a small open neighborhood $\mathcal{OB}(\mathcal A)$ with Vol$(\mathcal{OB}(\mathcal A))< 10^{-5l}$Vol$(M)$.  By Proposition \ref{gendirect} and the third item of Lemma \ref{finite}, $\mathcal M_{red}^{x, \Omega}$ are finite points in $\mathcal M^{x, \Omega}=S^2$ if $\Omega\in G_e^x$. Then we choose $r_0$ small enough such that $\mathcal M_{irr}^{x, \Omega}\cap \mathcal M_{irr}^{r_0}\neq \emptyset$ if $\Omega$ is chosen from the complement of $\mathcal{OB}(\mathcal A)$. 

Lemma \ref{4.14} is still valid in this situation with $\mathcal M$ replaced by $\mathcal M_{irr}^{r_0}$, which is Lemma 4.12 in \cite{T1}. We choose the constant $k_{r_0}$ as in Lemma 4.12 of \cite{T1} (or $s'_{r_0}=k_{r_0}^{-1}$ as in Lemma \ref{4.14}) and $s=k_{r_0}^{-4}$. Let $B_s(x)$ be the ball of radius $s$ centered at $x$.

We define a current $\Phi_e$ in the following manner.
Recall $$\mathcal M_{irr, l}=\{(C, x_1, \cdots, x_l)|C\in \mathcal M_{irr}, x_i\in M\}\subset \mathcal M\times M^{[l]}.$$

Use $\pi_l$ to denote the projection map $\mathcal M_{irr, l}\rightarrow M^{[l]}$. 
The portion of marked moduli space we choose is $\mathcal M^{r_0, r}_{irr,l}$,
 consisting of the set of marked curves with distance at least $r_0$ to 
$\mathcal M_{red}$ and $d(x_i, x_j)\ge r$ for any $i\neq j$. Here we suppose $r<\frac{s}{10}$. 

We first define $\phi_{\eta}(v)=\int_Cv$. Here $\eta \in \pi_l(\mathcal M^{r_0, r}_{irr,l})$, $C$ is the unique rational curve in $\pi_l^{-1}(\eta)$ and $v$ is a 2-form on $M$. Then we have the following {\it spherical current}
$$\Phi_e(v)=\int_{\eta\in \pi_l(\mathcal M^{r_0, r}_{irr,l})} \phi_{\eta}(v).$$

The spherical current $\Phi_e$ defined clearly satisfies Proposition 1.2 in \cite{T1}. Especially, it is a non-trivial, closed, non-negative $J-$invariant current on $M$.

\subsubsection{Estimates of the pencil  $\mathcal M^{ x_1, \Omega}$} \label{localpencil}

Fix $x_1$ and $\Omega\in G_e^{x_1}$. We write $\Omega=(x_3, \cdots, x_l)$. By Proposition \ref{gendirect}, $\mathcal M^{ x_1, \Omega}$ is a pencil. Moreover, by removing an open neighborhood of these finite directions corresponding to reducible curves, we can suppose the remaining directions correspond to the the curves with distance at least $r_0$ from $\mathcal M_{red}$.

We now assume $x_1$ and any entry of $\Omega$ are chosen from $M\setminus B_s(x)$ (but $\Omega$ not necessarily belongs to $G_e^x$). 
When the curve $C_{x, x_1, \Omega}$ is in $\mathcal M^{x_1, \Omega, r_0}$, with $x_1$ and $\Omega$  chosen from the compact set above (i.e. $d(x_i, x_j)\ge r$, $d(x_i, x)\ge s$ and $C_{x, x_1, \Omega} \in \mathcal M^{x_1, \Omega, r_0}$), there is a number $T>0$ such that for any $z\in \overline{B_T(x)}$, the sphere $C_{z, x_1, \Omega}$ is smooth and in $\mathcal M^{\frac{r_0}{2}}_{irr}$.
In other words, the part of $\mathcal M^{ x_1, \Omega}$ intersecting $\overline{B_T(x)}$ is a pencil of smooth curves. Clearly, this is also true for any $t\le T$. Let us denote this set by $\mathcal M^{x_1, \Omega ; B_t(x)}$. Notice the first defining condition of $\mathcal M_{irr, l}^{r_0, r}$ guarantees dist$(x, M(e))\ge r_0$ if $M(e)\neq \emptyset$.

Given a smooth curve $C$ in this pencil with normal bundle $N$, the tangent space to $\mathcal M^ {x_1, \Omega}_{irr}$ at $C$ can be identified with 
the vector space $\ker_{C, x_1, \Omega}\subset \Gamma(N)$ that consists of the sections in the kernel of $D_C$ 
that vanish at $x_1$ and $\Omega$.


On this two dimensional space, there are several norms. Let $\nu\in \ker_{C, x_1, \Omega}$.
\begin{itemize}
\item The $L^2$ norm  $||\nu||_2$;

\item The sup norm $\sup_C |\nu|$;

\item For $z\ne x_1$ or any entry of $\Omega$, the pointwise norm $|\nu(z)|$;

\item By choosing $x_1$, $\Omega$ as above, we could still define $\tau^{x_1, \Omega}$ or $\tau^{\Omega, x_1}$ by taking the complex direction $T_{x_1}C$ (or $T_{x_3}C$). Let $u$ denote the direction $T_{x_1} C$ in $\mathbb C \mathbb P^1$, 
$\tau_*^{x_1, \Omega}:T_C\mathcal M^{x_1, \Omega}\to T_u \mathbb C \mathbb P^1$
is an isomorphism. 
We thus could speak of the pointwise norm of  $\tau_*^{x_1, \Omega}\nu $ as a vector in $T_u \mathbb C \mathbb P^1$. 
\end{itemize}

For a fixed curve  $C$, these norms are equivalent. 
Since $\mathcal M_{irr, l}^{r_0, r}$ is compact, if we have compact families of choices of $x_1, \Omega, z$, 
we have uniform constants as in Lemma \ref{dist}.

\begin{lemma}\label{distl} 
Let   $x$, $x_1$, $\Omega$ and $C$ be as stated as above. There is a constant $k_{\ref{distl}}>1$, depending only on $r_0$, $r$, $s$ and $T$,  ensuring the following inequalities for $\nu\in \ker_{C, x_1, \Omega}$:
\begin{enumerate}
\item 
$\sup_C|\nu|\le k_{\ref{distl}}  |\nu(z)|$ if  $x_1$ and any entry of $\Omega$ are not in $B_s(x)$ and $z\in\overline{B_{\frac{s}{2}}(x)}$;
\item $\sup_C|\nu|\le k_{\ref{distl}} |\tau^{x_1, \Omega}_{*} \nu|$  if  $x_1\in \overline{B_s(x)}$, and $\sup_C|\nu|\le k_{\ref{distl}} |\tau^{\Omega, x_1}_{*} \nu|$  if  $x_3\in \overline{B_s(x)}$.

\end{enumerate}

\end{lemma}

\begin{proof}
 The constant in (1) can be chosen to be independent of $x$, $x_1$,  $\Omega$, $z$, and $C$ since $M$, $M\setminus B_s(x)$, $M\setminus B_r(x_i)$,  $\overline{B_{\frac{s}{2}}(x)}$ and $\mathcal M^{x_1, \Omega; B_T(x)}$
are compact.

The constant in (2) is uniform because $(x_1, \Omega)$ is chosen from a compact set in $M^{[l-1]}$.
\end{proof}

Similarly, Lemmas \ref{outside} and \ref{angle-volume} are also valid with apparent modification in the statement.

Let $T^{x_1, \Omega;B_t(x)}$ denote the set of points $x_2$ in $M$ that lies in a curve in $\mathcal M^{x_1, \Omega}$ and intersecting $\overline{B_t(x)}$.

\begin{lemma} \label{outside'}
Suppose $x_1$ and each entry of $\Omega$ are not in $B_s(x)$. 
There are constants $k$ and $\kappa$ 
depending on $s$ with the following significance:
For $t<\kappa^{-3}$, 
the volume of $T^{x_1, \Omega;B_t(x)}$  is bounded from above by $kt^2$. 

\end{lemma}

\begin{proof}
The proof is similar to that of Lemma \ref{outside}. Let $C=C_{x, x_1, \Omega}$.   
Since dist$(x_1, x)\geq s$ and dist$(\Omega, x)\ge s$, $|\nu(x)|$ is a norm on the $2-$dimensional vector space ker$_{C, x_1, \Omega}$. 
Now $\mathcal M^{x_1, \Omega; B_t(x)}$ is a $2-$dimensional smooth compact manifold. As argued in Lemma \ref{outside}, $$\hbox{dist}(x_2, C) \le  2K\kappa_{\ref{distl}}\kappa t.
$$
 Then the volume of $T^{x_1, \Omega;B_t(x)}$  is bounded from above by $kt^2$.
\end{proof}

\begin{lemma}\label{angle-volume'}

Suppose  $w=x_3$ is in $\overline{B_s(x)}\setminus B_{Rt}(x)$. There is a constant $k$ depending on $s$ and $R$ with the following significance:

The volume of points $x_2$ in $M$ lying in a curve in $\mathcal M^{x_1, \Omega, r_0}$ and intersecting $B_t(x)$ is bounded from above by $k\frac{t^2}{d^2}$, where $d=dist(w, x)$.

\end{lemma}

\begin{proof}
The proof is identical to that of Lemma \ref{angle-volume}, 
with the discussion on the map $\tau^{\Omega, x_1}$ (notice $w=x_3$ is 
the first entry of the superscript) and  Lemma \ref{distl}(2) in place of Lemma \ref{dist}(2).
\end{proof}

\subsubsection{Upper bound for a big $J-$nef class}
Now we denote the center of $B$ in definition \ref{tcdef} by $x$ where $x$ is any point in $M$ and denote the ball by $B_t(x)$.

\begin{prop} \label{upper}
Let  $e$ be a big $J-$nef  class in $S_{K_J}$.  
The current $\Phi_e$  satisfies the upper bound in \eqref{leqgeq}.
\end{prop}

\begin{proof} 
Choose $s$ as in the beginning of this section.

Fix $x\in M$ and adapted coordinates $(z, w)$ centered at $x$ with radius $s$. 

Let $0<t<10^{-5}  s$.

As in \cite{T1}, we only need to prove  $\Phi_e(i\chi_tdz\wedge d\bar z)<kt^4$. 
 Let us denote and group the $l$ points by $x_1$ and $x_2$ and $\Omega=(x_3, \cdots, x_l)$. Moreover, $\Phi_e(i\chi_tdz\wedge d\bar z)$ is no greater than
$$k_0\int_{\eta=(x_1, x_2, \Omega)\in \pi_l(\mathcal M_{irr, l}^{r_0, r})}(\int_{ C_{\eta}}\chi_t\omega)
$$

Notice that  $\int_{ C_{\eta}}\chi_t\omega=0$
if  $C_{\eta}\cap B_t(x)= \emptyset$.   
If  $C_{\eta}\cap B_t(x)\neq \emptyset$, then $C_{\eta}\cap B_t(x)$ is contained in a ball
of radius of $2t$ centered at some point in $B_t(x)$. 
By Lemma \ref{2.2},   the integrand $\int_{ C_{\eta}}\chi_t\omega$ is bounded by $H\cdot [\omega] k t^2$.

Thus it suffices to prove that the volume of the set
\begin{equation} \label{volume of pairs} \{\eta=(x_1, x_2, \Omega)\in \pi_l(\mathcal M_{irr, l}^{r_0, r}),  C_{\eta}\cap B_t(x)\neq \emptyset\}
\end{equation}
is $O(t^2)$.

We choose $\Omega$, $x_1$ and $x_2$ in turns. 

We have three cases depending on the positions of $\Omega$ and $x_1$. 

I. The first case is that $x_1$ and each element of $\Omega$ are all away from $B_s(x)$. 

Since our moduli space for integration is $\mathcal M_{irr, l}^{r_0, r}$, and $x_1\in M\setminus B_s(x)$, $\Omega\in \{(x_3, \cdots, x_l)| x_i\in M\setminus B_s(x)\}$.  The corresponding upper bounds for these two factors are $\hbox{Vol}(M)$ and $\hbox{Vol}(M)^{l-2}$ respectively.

For those $\Omega$ and $x_1$ contributing to the integration, 
$\mathcal M^{\Omega, x_1, x'}\subset \mathcal M^{r_0}_{irr}$ for some $x'\in \overline{B_t(x)}$. 
 Thus by choosing $t<T$, for any $z\in \overline{B_t(x)}$, the unique smooth rational curve $C_{\Omega, x_1, z}\in \mathcal M^{\frac{r_0}{2}}_{irr}$.

Now we estimate the possible choices of $x_2$. 
By the above picture, the part of $\mathcal M^{\Omega, x_1; B_t(x)}$ is a pencil.
By Lemma \ref{outside'}, the volume of $x_2$ is bounded from above by $k_1t^2$. This constant $k_1$ could be chosen uniformly since the possible set of $\Omega$ and $x_1$ is a closed, then compact, subset of $M^{[l-1]}$. The factors $\hbox{Vol}(M)^{l-2}$, $\hbox{Vol}(M)$ and $k_1t^2$, multiply to an upper bound of $O(t^2)$ for the volume of the subset in (\ref{volume of pairs}) with $x_1$ and each element of $\Omega$ in $M\setminus B_s(x)$.

II. The second case is when $x_1$ or some entry of $\Omega$, say $w$, satisfies $Rt<\hbox{dist}(w,x)<s$, where  $R=10^5$. The proof is almost identical to part II of Proposition \ref{upper'} but invoking Lemma \ref{angle-volume'} instead.

III. The last case is when any entry of $\Omega$ or $x_1$ is in $B_{Rt}(x)$. This is exactly the last case of Proposition \ref{upper'}. 

Summing  the three cases, we finish the proof.
\end{proof}



\subsubsection{Taubes current from  a $J-$ample class}


\begin{prop} \label{positive} Let $e$ be a $J-$ample  class $e$ in $S_{K_J}$. 
Then  the current $\Phi_e$ is a Taubes current, i.e. it  satisfies \eqref{leqgeq}.
Consequently,  there is an almost K\"ahler form in the same class.
\end{prop}

\begin{proof}

Thanks to Proposition \ref{upper}, we
  only prove the lower bound $k^{-1}t^4<\Phi_e(f_{B_t(x)}i\sigma\wedge \bar\sigma)$ here. As in \cite{T1}, we prove  $k^{-1}t^4<\Phi_e(if_{B_t(x)}dz\wedge d\bar z)$. Let us denote the $l$ points by $x_1$, $x_2$ and $\Omega=(x_3,\cdots, x_l)$. If $l=2$, we only have $x_1$ and $x_2$. 
  Let  $0<t<10^{-5}s$.

The main picture to have is Lemma \ref{curves in a ball}, applied to $B_s(x)$. Namely, inside $B_s(x)$, the curves behave as straight lines with respect to the adapted coordinates. 

Since we are estimating the lower bound, in addition to choosing $\Omega$ outside a 
small open neighborhood $\mathcal{OB}(\mathcal A)$ of the measure zero set $\mathcal A$,
 we also know that each entry of $\Omega$ is away from $B_{\frac{s}{2}}(x)$ and each entry
 of $\Omega$ is at least of distance $r$ from each others. This set of $\Omega$ is compact in 
$M^{[l-2]}$. To summarize, by choosing $s$ small, all such $\Omega$ constitute a compact set of volume no smaller than $$(\hbox{Vol}(M)-(l-2)\max_{x\in M} \hbox{Vol}(B_{\frac{s}{2}}(x)))^{l-2}-\hbox{Vol}(\mathcal{OB}(\mathcal A))>(\frac{\hbox{Vol}(M)}{2})^{l-2}.$$

 Before making choices of $x_1$ and $x_2$, we digress to choose a compact submanifold $\mathcal C_{x, \Omega}\subset \mathcal M^{x,\Omega}_{irr}$.
 
 By Proposition \ref{gendirect} and Lemma \ref{finite}, for a pretty generic $\Omega$, except for finitely many complex directions in $T_xM$, 
there is a smooth rational curve passing through this direction and $\Omega$. Recall that $\mathcal M^{x, \Omega}$ is a pencil, and identified with $\mathbb C \mathbb P^1$ via $\tau^{x, \Omega}$. Then the set $\mathcal C_{x, \Omega}$ is characterized by the following two properties: 
\begin{itemize}
\item Its image under $\tau^{x, \Omega}$ is contained in a disk $|dz|(\cdot)\ge 2\epsilon$ in $\mathbb C \mathbb P^1$;
\item Any curve $C\in \mathcal C_{x, \Omega}$ has distance at least $r$ from $\mathcal M_{red, e}$.
\end{itemize}

It is a (nonempty) compact submanifold of $\mathcal M^{x, \Omega, r_0}_{irr}$ (of real dimension two). Recall we choose $r_0$ small enough such that $\mathcal M^{x, \Omega, r_0}_{irr}=\mathcal M_{irr}^{x, \Omega}\cap \mathcal M_{irr}^{r_0}\neq \emptyset$. Moreover, for $C$ chosen from $\mathcal C_{x, \Omega}$, the restriction of $|dz|$ to $C\cap B_s(x)$ is greater than $\epsilon$ when $s$ is chosen sufficiently small.

Now, let us choose $x_1$. Again, we choose $x_1\in B_s(x)$ away from $B_{\frac{s}{2}}(x)$ and $B_{\frac{s}{2}}(x_i)$ where $x_i$'s are entries of $\Omega$. Additionally, we choose $x_1$ such that the rational curve $C_{x_1, x, \Omega}$ determined by $x$, $\Omega$ and $x_1$ is contained in the compact submanifold $\mathcal C_{x, \Omega}\subset \mathcal M^{x, \Omega}_{irr}$ specified above.  By Lemma \ref{finite},  the choices of $x_1$ constitute a compact set with nonzero volume, say $c_{\Omega}s^4$.

Now, with  $x_1$ fixed, we consider the set of $x_2$ in $B_{\frac{s}{4}}(x)$, 
for which $\eta=(x_1, x_2, \Omega)$ contributes to $\Phi_e(if_{B_t(x)}dz\wedge d\bar z)$,
 namely $C_{x_1, x_2, \Omega}$ intersects $B_t(x)$. This part of argument is 
identical to the corresponding part in Proposition \ref{positive'} with $C_{x_1, x_2, \Omega}$ and $\pi_l$ 
appearing in  place of $C_{x_1, x_2}$ and $\pi_2$. We have the lower bound $k_1s^2t^2$. 

Let $C=C_{x_1, x_2, \Omega}$, then $\int_C(if_{B_t(x)}dz\wedge d\bar z)\ge k_2^{-1}t^2$ by Lemma \ref{2.2} as in Proposition \ref{positive'}.

Multiplying these four factors: $(\frac{\hbox{Vol}(M)}{2})^{l-2}$, $c_{\Omega}s^4$, $k_1s^2t^2$ and $k_2^{-1}t^2$ together, we get $k^{-1}t^4<\Phi_e(if_{B_t(x)}dz\wedge d\bar z)$. These constants 
$c_{\Omega}$, $k_1, k_2$ are  independent of $t$ as one can check from the proof. 
The constant $k$ could be chosen universal by the same reasoning in Proposition \ref{positive'}.

The last statement follows from Proposition \ref{C-F}.
\end{proof}


\subsubsection{Weak Taubes current from a big $J-$nef class}

If $e$ is big $J-$nef but not $J-$ample, $\Phi_e$ is not a Taubes current, since 
   no irreducible curves in class
$e$ pass through  points in the vanishing locus $Z(e)$.  
Nonetheless, the following observation will be very useful.

\begin{prop} \label{non-negative} Let  $e$ be a big $J-$nef class.
Then the current $\Phi_e$ is non-negative, and over any ($4-$dimensional) compact submanifold $K$ of the complement $M(e)$, it satisfies \eqref{leqgeq} for a constant $k>1$ depending only on $K$. 
\end{prop}

\begin{proof} $\Phi_e$ is a non-negative current 
by definition, and the upper bound is from Proposition \ref{upper}.  
We only need to prove that it is bounded from below by $k^{-1}t^4$ 
on any compact submanifold $K\subset M(e)$. The proof  is almost identical to the proof of 
Proposition \ref{positive}. 
 Notice that all the relevant results in Section $4$ are 
established for $x\in M(e)$. Hence the proof goes almost verbatim as that of Proposition \ref{positive}
 when $x$ is chosen from a compact submanifold $K\subset M(e)$, and $\Omega$, $x_1$ chosen from 
$K'^{[l-2]}$, $K'$ (instead of from $M^{[l-2]}$ and $M$) respectively. Here $K'$ is another ($4-$dimensional) compact submanifold of $M(e)$ such that $K\subsetneq K'$, which satisfies $B_s(p)\subset K'$ if $p\in K$.
\end{proof}

We call a current in Proposition \ref{non-negative} a {\it weak Taubes current}. By summing up weak Taubes currents with disjoint zero locus, we obtain Taubes currents.

\begin{prop}\label{form}
Let $e_i$ be big $J-$nef  classes in $S_{K_J}$ and denote $Z_i$ the zero locus of $e_i$. 
If $\cap Z_i =\emptyset$, then there is a Taubes current in the class $e=\sum_i a_ie_i$, with $a_i>0$. In turn, 
we obtain an almost K\"ahler form in the  class $e$.
\end{prop}


\subsection{Tamed versus compatible}


\subsubsection{Proof of Theorem \ref{s2bundle-1}}

\begin{proof} Given any  $J$ on $S^2-$bundle over $S^2$ 
it follows from Lemma \ref{s2bundle-nef} that there is always a $J-$ample class. Now the conclusion follows from 
Proposition \ref{positive}. 
\end{proof}
\subsubsection{Tameness and foliation}

Now we have shown that every tamed $J$ on an $S^2-$bundle over $S^2$ is almost K\"ahler,
a related question is when  $J$ is tamed.  

Suppose $M$ is an $S^2-$bundle over $S^2$ and $J$ is an almost complex structure on $M$, not necessarily tamed. 
Then there is a $J-$nef class in $S^0_{K_J}$.  
If $J$  is tamed, by  Proposition \ref{existence},  $e$  is $J-$effective, and moreover, there is a $J-$holomorphic
foliation by smooth rational curves. We would like to know whether the converse is also true.

\begin{question}
For an $S^2-$bundle over $S^2$,  suppose   there is a $J-$holomorphic
foliation by smooth rational curves.    
Is $J$ tamed?

\end{question}

Note that if we assume further that $J$ is fibred, Gompf's construction in  \cite{Gpf} produces a tamed symplectic form.  

Here is an analogous question for $\mathbb C \mathbb P^2$.

\begin{question}
For $\mathbb C \mathbb P^2$,  
$J$ is tamed if and only if there is a pencil of smooth rational curves?
\end{question}


\subsubsection{$\mathbb C \mathbb P^2\# k\overline {\mathbb C \mathbb P^2}$ with $k\geq 2$}
From now on  we sometimes denote  $\mathbb C \mathbb P^2\# k\overline {\mathbb C \mathbb P^2}$ by $M_k$.
For $M_k$ with $k\geq 2$, there are no $J-$ample classes in $S_{K_J}$.  So we apply  Proposition \ref{form} to construct Taubes currents.

Let $H, E_i$ be an orthogonal basis of $H_2(M_k;\mathbb Z)$ with
$H^2=1, E_i^2=-1$. Such a basis is called a standard basis.

Given any tamed almost complex structure $J$, there is a standard basis 
such that $K_J=-3H+\sum E_i$, called a standard basis adapted to $J$. 
This follows from the uniqueness of symplectic canonical classes, up to diffeomorphisms. 
Notice that  $E_i\in\mathcal E_{K_J}$ and  $H\in S^+_{K_J}$.
Hence $H$ is $J-$effective. And it pairs positively with any $J-$tamed symplectic form. 

Suppose $J$ is tamed and  there is a configuration  of $k$ disjoint  $-1$ curves $C_i$.
Notice that we have a standard basis $(H, E_1, \cdots, E_i)$ adapted to $J$, with $E_i=[C_i]$ and 
$H$  the unique square 1 class with $H\cdot K_J<0$ and $H\cdot E_i=0, i=1, \cdots, k$.

\begin{lemma} \label{package}
Suppose $J$ is tamed and  there is a configuration  of $k$ disjoint  $-1$ curves $C_i$.
Then  the classes $H$, $2H$, $nH-(n-1)E_i$,  $nH-(n-1)E_i-E_j$ and $H-E_i$ in $S_{K_J}$ are $J-$nef.
\end{lemma} 
\begin{proof} Let $C$ be an irreducible curve distinct from any $C_i$, and suppose 
$[C]=aH-\sum_ib_iE_i$.
Since  $C_i$ is  irreducible,  $b_i\ge 0$ by the positivity of intersection.  Since any $J-$tamed $\omega$ is positive on $H$, 
 $a>0$.

Clearly $H, 2H$ are $J-$nef since $a>0$. 

Since   $C$  is  an irreducible curve, $g_J(e_C)\geq 0 $, so we have

  \begin{equation}\label{general}\begin{array}{lll}
-2&\leq& 2g_J(e_C)-2\cr
&=&K_J\cdot C+C^2\cr
&=&-3a+\sum_ib_i+a^2-\sum_i b_i^2\cr
&= &(a-b_1)(a+b_1-1)\cr
&&-2a\cr
&&-\sum_{i\geq 2} (b_i^2-b_i).\cr
\end{array}
\end{equation}
The second term is at most $-2$.
The third term is non-positive. 
Therefore $a\geq b_1$. 
In fact, the same argument shows that  $a\geq b_i$ for any $i$. 

It follows that  $C$ cannot  pair  negatively with $nH-(n-1)E_i$,  $H-E_i$. 

If  $C$  pairs  negatively with $nH-(n-1)E_i-E_j$, then 
$$na-(n-1)b_i-b_j<0,$$ which implies that either $a<b_i$ or $a<b_j$. 
But this is impossible so $nH-(n-1)E_i-E_j$ is also $J-$nef. 
\end{proof}

\begin{remark}\label{no smooth curves}
Under the assumption of Lemma \ref{package}, it is not true that  any class in $S_{K_J}^{\geq 0}$ is  $J-$nef. Consider the holomorphic  blow up 
of 3 points on a line $l$  in $\mathbb C \mathbb P^2$ and let $C_1, C_2, C_3$ be the exceptional curves. 
Then the class $2H-E_1-E_2-E_3$ in $S_{K_J}$ is not $J-$nef since it pairs negatively 
with the class of the proper transform $l'$  of the line $l$, which is  $H-E_1-E_2-E_3$.  
In particular, there are no smooth rational  curves in the  class $2H-E_1-E_2-E_3$.
\end{remark}

\begin{remark}\label{singular foliation}
Notice that, under the assumption of Lemma \ref{package}, there is at least one $J-$nef class in $S_{K_J}^0$.  
It is natural to wonder whether this is true for any tamed almost complex structure.
If so, by Proposition \ref{1.1},   every tamed rational $4-$manifold has a  rational curve `foliation', with only finitely many reducible leafs. 
\end{remark}

We are ready to prove  Theorem \ref{k curves}.

\begin{proof}

  Observe that,  for each $i$, $2H-E_i$ is a  $K_J-$spherical class with square $3$, and it is $J-$nef by Lemma \ref{package}.

Let $Z_i=Z(2H-E_i)$ be the zero locus of $2H-E_i$.
If   $C\ne C_i$  is an irreducible curve   in $Z_1$,  then $b_1=2a>a$. 
This is impossible by the proof of Lemma \ref{package}.
Therefore
$$Z_i=\{C_1,  \cdots,  C_{i-1}, C_{i+1},\cdots\}.$$
Clearly, $\cap_{1\leq i\leq k} Z_i =\emptyset$. 

Therefore 
 there is an almost K\"ahler form in the class $\sum_{i=1}^k (2H-E_i)$ by Proposition \ref{form}.
 In particular,  $J$ is almost K\"ahler.
\end{proof}

 \begin{cor}\label{tcgood}
If $J$ is in $\mathcal J_{top}$ or $\mathcal J_{good}$, then $J$ is almost K\"ahler. 
\end{cor}

\begin{proof}

For any tamed $J$, by Proposition \ref{existence}, each $E$  in $\mathcal E_{K_J}$ is represented by a $J-$holomorphic subvariety $\Theta$.
We claim that $E$ pairs non-negatively with any irreducible subvariety $C$ whose class is not $E$.
This is clear if $C$ has non-negative self-intersection. Suppose $C$ has negative self-intersection. Then by our assumption on $J$, 
$C$ is either a $-1$ curve or or the anti-canonical curve. 
The $-1$ curve case follows from 
Lemma \ref{pairS+} (3), and the anti-canonical curve case follows from the adjunction formula.

The following result is proved in \cite{LZrc} as Proposition 4.25:
\begin{prop}\label{connected}
Suppose $J$ is tamed, $e\in S_{K_J}$ and $\Theta=\{(C_i, m_i)\}\in \mathcal M_e$. If $e\cdot e_{C_i}\ge 0$, then $\Theta$ is connected and each component $C_i$ is a smooth rational curve.
\end{prop}

Our $\Theta$ satisfies the condition, so $\Theta$ is connected and  each component is a smooth rational curve. 
  Since $\Theta$  is connected,  if it is reducible,  it must contain an irreducible component $F$ with self-intersection at most $-2$.
But by our assumption on $J$ there are no smooth rational curves of self-intersection less than $-1$.

We have shown that each $E$  in $\mathcal E_{K_J}$ is represented by a $-1$ curve. In particular, 
given a standard basis $\{E_i\}$ adapted to $J$,  there are $k$ disjoint $-1$ curves in the $k$ classes $E_i$.  
Now apply Theorem \ref{k curves}. 
\end{proof}

\begin{cor}\label{standardform} Suppose  $h_l=lH-\sum_{i=1}^k E_i$ with $l^2>\max\{k,9\} $.
 If $J$ is an almost complex structure on $M_k$   tamed by a symplectic form in the class $h_l$, then 
 $J$ is almost K\"ahler. Moreover, there is an almost K\"ahler form in the class $h_l$. \end{cor}

\begin{proof} Since the $h_l$ area of $E_i$ is 1, a subvariety representing  $E_i$ is  irreducible, and hence it is smooth. Thus $J$ is almost K\"ahler by Theorem \ref{k curves}. 

Since $h_l$ is in the  $J-$tamed cone $\mathcal K_J^t$ and $J$ is almost K\"ahler, the last claim follows from the equality   \eqref{tc} between $\mathcal K_J^t$ and the almost K\"ahler cone. 
\end{proof}

We call a tamed almost complex structure $J$   del Pezzo if
 it is tamed by a symplectic form in the class $-K_J$. By Corollary \ref{standardform}, del Pezzo $J$ is almost K\"ahler.

\begin{remark} It is  observed by    Pinsonnault \cite{Pinn}  
 that for  any tamed $J$ on $M_k$, there exists   at least one (smooth) $-1$ curve. 
 \end{remark}


\subsection{Almost K\"ahler cone}

We first introduce an open convex cone associated to $J$. 
\begin{definition}
For a tamed almost complex structure $J$ on $M_k$ (resp. $S^2\times S^2$), the open convex  cone $\mathcal S_J$ is defined to be the interior of the  convex 
cone generated by big $J-$nef classes in $S_{K_J}$ if it is of dimension $k+1$ (resp. $2$). Otherwise, it is defined as $\emptyset$. 
\end{definition}

For an almost K\"ahler $J$, we have the following

\begin{lemma}\label{convenient}
If $J$ is almost K\"ahler, 
then $\mathcal S_J$
is contained in  $\mathcal K_J^c(M)$.
\end{lemma}

\begin{proof}
Let $\mathcal S_J\neq \emptyset$. Suppose $u$ is the class of an almost K\"ahler form. 
Observe that given any class $e \in \mathcal S_J$, $e-tu$ is in $\mathcal S_J$ for $t$ small. Thus $e-tu=\sum a_ie_i$ with $a_i>0$ and $e_i$ big $J-$nef . Hence $e=(e-tu)+tu\in \mathcal K_J^c(M)$ by 
Propositions \ref{non-negative} and \ref{C-F}.
\end{proof}

\subsubsection{$S^2-$bundles over $S^2$}

 \begin{proof}[ Proof of 
 Theorem  \ref{s2bundle-2}]  
 By Lemma \ref{convenient} it suffices to show that
  $$\mathcal S_J=\mathcal P_J$$
 
 We establish this by describing explicitly the curve cone $A_J$. 
 We have mentioned that there is always
 a foliation by smooth rational curves. 
  One boundary of the curve cone is generated by
 the class of such a foliation. 
 The other boundary is generated by a transversal class.
 
 We start with  $S^2\times S^2$.

Given any $J$, denote the class of a foliation by $H_2$. 
The curve cone $A_J$ is generated by  $H_2$ and a transversal class $H_1-lH_2$ for some $l\geq 0$.

The $\geq 0-$dual of $A_J$ is generated by $A=H_1+lH_2$ and $B=H_2$.

If $l=0$, then $A=H_1$ and it is approximated by the sequence of big $J-$nef classes  in $S_{K_J}$,  $pH_1+H_2$. $B$ is approximated by the sequence of big $J-$nef classes  in $S_{K_J}$, $H_1+qH_2$. Therefore
$$\mathcal S_J=\mathcal P_J=\{aA+bB|a>0, b>0\}=\{xH_1+yH_2|x>0, y>0\}.$$

If $l>0$, then $A$ itself is a big $J-$nef class in $S_{K_J}$, and $B$ is approximated by the sequence of big $J-$nef classes  in $S_{K_J}$, $H_1+qH_2$, $q\geq l$.
Therefore $$\mathcal S_J=\mathcal P_J=\{aA+bB|a> 0, b> 0\}=\{xH_1+yH_2|y> lx>0\}.$$

For $M=\mathbb C \mathbb P^2 \# \overline{\mathbb C \mathbb P^2}$, the proof is similar.

Given any $J$, the unique class of foliation is  $H-E$. 
The curve cone $A_J$ is generated by   $H-E$ and a transversal class $D_{-l}=-lH+(l+1)E$ for some $l$.

The  $\geq 0-$dual of $A_J$ is generated by $B=H-E$ and $A=(l+1)H-lE$.

For each $l$, the class  $A$ is a big $J-$nef class in $S_{K_J}$,    and $B$ is approximated by the  sequence of big $J-$nef classes  in $S_{K_J}$,      
$(p+1)H-pE, p\geq l$.
Thus $$\mathcal S_J=\mathcal P_J=\{aA+bB|a> 0, b>0\}=\{ xH-yE|\dfrac{l+1}{l}y> x>y>0\}.$$

\end{proof}

\subsubsection{$\mathbb C \mathbb P^2\# k\overline {\mathbb C \mathbb P^2}$}
In this case we will  again apply Lemma \ref{convenient}  to probe the almost K\"ahler cone.

In fact, for $J\in  \mathcal J_{top}$, Question \ref{aknm} takes a particular simple form, which we now explain. 
Among all $J\in \mathcal J^t$, $J$ in the top stratum  $ \mathcal J_{top}$ has the maximal $\mathcal P_J$. In fact, when $b^+(M)=1$ and $J\in \mathcal J_{top}$, $\mathcal P_J$ is 
 equal to  the  $K_J-$symplectic cone $\mathcal C_{ K_J}$ introduced in \cite{LL}:
  \begin{equation}\label{Kcone}
\mathcal C_{ K_J}=\{e\in H^2(M;\mathbb R)|e=[\omega] \hbox{ for some $\omega$ with $K_{\omega}=K_J$}\}.
\end{equation}
Here $K_{\omega}$ is the symplectic canonical class of $\omega$. 
Thus  the almost K\"ahler Nakai-Moishezon criterion for  $J\in  \mathcal J_{top}$
 is the same as   $$\mathcal K^c_J(M)=
\mathcal C_{ K_J}.$$

Although we cannot verify the almost K\"ahler Nakai-Moishezon criterion for  $J\in  \mathcal J_{top}$ when $k\geq 10$, we have  the following partial  result. 

 \begin{theorem} \label{sc}
Suppose $M$ is a rational manifold.
 If  $J\in \mathcal J_{top}$, then $\mathcal K^c_J(M)\supset 
\mathcal C_{ K_J}\cap \{e|e\cdot K_J<0\}.$
\end{theorem}

To establish Theorem \ref{sc} we introduce several open convex cones associated to $K_J$.

\begin{definition}\label{sym-spherecone}
For a tamed almost complex structure $J$, the open convex cone   $\mathcal S^+_{K_J}$  is the interior of the  convex cone generated by classes in $S^+_{K_J}$.
It is called the positive $K_J-$sphere cone. 
\end{definition}

Notice that
 $\mathcal S^+_{K_J}$ contains  $\mathcal S_{J'}$ 
for any tamed $J'$ with $K_{J'}=K_J$.

According to \cite{LL},  
the  $K_J-$symplectic cone $\mathcal C_{ K_J}$ has the following characterization, 
\begin{equation}\label{Kcone'}
\mathcal C_{K_J}=\{e\in \mathcal P|e\cdot E>0 \hbox{ for any }  E\in \mathcal E_{M, K_J}\}.
\end{equation}
Finally, introduce the open subcone of $\mathcal C_{K_J}$,
\begin{equation}\label{p-k}
P_{K_J}:=\{e\in \mathcal C_{K_J}|e\cdot (-K_J)> 0\}
\end{equation}

\begin{prop} \label{S=P} For any $J$ on $\mathbb C \mathbb P^2\# k\overline {\mathbb C \mathbb P^2}$, 
the positive $K-$sphere cone $\mathcal S^+_{K_J}$ coincides with $P_{K_J}$. 
\end{prop}
We will defer the proof 
of Proposition \ref{S=P} to the next subsection, where we need to review $P-$cell in \cite{FM} and $K-$symplectic cone in \cite{LL}.

\begin{lemma}\label{J=K}
If $J$ is in $\mathcal J_{top}$ or $\mathcal J_{good}$, then $\mathcal S_J=\mathcal S_{K_J}^+$.
\end{lemma}
\begin{proof}

When $J$ is in $\mathcal J_{top}$ or $\mathcal J_{good}$, by Lemma \ref{package} and the proof of Corollary \ref{tcgood},    any class in $S_{K_J}^+$ is $J-$nef.

\end{proof}

\begin{proof}[Proof of  Theorem \ref{sc} and Theorem \ref{good}]  
Both claims  follow
 from Lemmas \ref{convenient}, \ref{J=K}, Proposition \ref{S=P}, and \eqref{p-k}.
\end{proof}

 It can be easily shown that  if $J$ is del Pezzo then  $J\in \mathcal J_{top}$, thus we have

\begin{cor}\label{cor:del Pezzo}
 If $J$ is del Pezzo, then $\mathcal K_J^c(M)=\mathcal C_{K_J}$.
 In other words,  the almost K\"ahler cone either does not contain $-K_J$, or is equal to $\mathcal C_{K_J}$.
\end{cor}

This result also follows from the $J-$inflation approach in  \cite{Z}.
The $J-$inflation along a smooth $J-$holomorphic subvariety,  which is due to  McDuff  \cite{Mc2} and extended by Buse \cite{Bus}, can be effectively 
applied to probe the $J-$tamed cone 
of a tamed almost complex structure $J$,
$$
\mathcal K_J^{t}=\{[\omega]\in H^2(M;\mathbb R)|\hbox{$\omega$  tames $J$}\}.
$$
Clearly,  $\mathcal K_J^c\subset \mathcal K_J^t$.
When $b^+=1$, it was shown in \cite{LZ} that
 if  $J$ is almost K\"ahler, then
 \begin{equation}\label{tc}
 \mathcal K_J^c=\mathcal K_J^t.
 \end{equation}
 The equality \eqref{tc} combined with the calculation of the tamed cone in \cite{Z} 
 gives an alternative proof of Corollary \ref{cor:del Pezzo}. In fact, 
   in  \cite{Z}  we use this approach to establish the almost K\"ahler Nakai-Moishezon criterion  for minimal ruled manifolds  and  all rational manifolds
with $b^-\leq 8$.


\subsection{$K-$sphere cone and $K-$symplectic cone}
In this subsection we establish Proposition \ref{S=P}. 
\subsubsection{$P-$cells}

Suppose $M$ is an oriented closed manifold with odd intersection form, $b^+=1$, $b^-=n$
and no torsion in $H^2(M;{\mathbb Z})$. A basis
$(x, \alpha_1, \cdots, \alpha_n)$ for $H^2(M;{\mathbb Z})$ is called
standard if $x^2=1$, and $\alpha_i^2=-1$ for each $i=1, \cdots, n$.
Let $$\begin{array}{ll} {\mathcal P}&=\{e\in H^2(M;{\mathbb R})|e\cdot e>0\}\cr
{\mathcal B}&=\{e\in H^2(M;{\mathbb R})|e\cdot e=0\}\cr
{\overline {\mathcal P}}&=\{e\in H^2(M;{\mathbb R})|e\cdot e\geq 0\}.\cr
\end{array}$$
For each class $x\in H^2(M;{\mathbb Z})$ with $x^2<0$,
we define $x^{\perp}\in H^2(M;{\mathbb R})$ to be the orthogonal
subspace to $x$ with respect to the cup product, and we call
$(x^{\perp})\cap {\mathcal P}$ the wall in ${\mathcal P}$ defined by
$x$.
Let ${\mathcal W}_1$ be the set of walls in ${\mathcal P}$ defined by
integral classes with square $-1$. A chamber for ${\mathcal W}_1$
is the closure in ${\mathcal P}$ of a connected component of
${\mathcal P}-\cup_{W\in {\mathcal W}_1}W.$

Any point $x\in {\mathcal P}$ with square 1  at which  $n$ mutually perpendicular
walls of ${\mathcal W}_1$ meet is called a corner.
Any corner is an integral class (see Lemma 2.2
in \cite{FM}).
Suppose $C$ is a chamber for ${\mathcal W}_1$. If $x$ is a corner
in $C$, a standard basis $(x, \alpha_1,\cdots, \alpha_n)$
for $H^2(M;{\mathbb Z})$ is called a standard basis adapted to $C$ if
$\alpha_i\cdot C\geq 0$ for each $i$. The canonical class
of the pair $(x, C)$ is defined
to be $\kappa(x, C)=3x-\sum_i \alpha_i$.
Suppose $C$ is a chamber for ${\mathcal W}_1$
 and $x$ is a corner in $C$,  we define
$$P(x, C)=C\cap \{e\in {\mathcal P}|\kappa(x, C)\cdot e\ge 0\}.$$
Any subset of ${\mathcal P}$ of the form $P(x, C)$ is called a $P-$cell.

\subsubsection{$\mathcal C_K$ and $P_K$} 
We are back to the situation that  $M=M_k$  and $J$ is a tamed almost complex structure.  Denote $K_J$ by $K$.

By  Lemma 2.4 in \cite{LBL},
 $\overline P_K$, the closure of $P_K$,
is a $P-$cell and
$\kappa(P_K)=-K$.

\begin{lemma}\label{face}
1. $P_K$ is an open convex polytope in $\mathcal P$. Each wall of $P_K$ is 
either a wall of a class in $\mathcal E_{M, K}$, or the wall of $K$ if $k\geq 9$.

2. The face $F_{E_k}$ of $P_{M_k, K}$ corresponding to  $E_k$ is naturally identified with $P_{M_{k-1}, K}$.
\end{lemma}

\begin{proof}
 Statement 1 is due to Friedman and Morgan. It states that $P_K$ does not have round boundary, i.e.
 boundary contributed by $\mathcal B$, although $\mathcal C_{K}$ always has a round boundary when $k>9$.
It is easy to see that 
when $k\leq 8$, $P_K$ is just a chamber. In other words, each wall of it is a wall of a class in $\mathcal E_{M, K}$. When $k>  9$, the class $K$ does contribute a wall to $P_K$.

For 2, $E_k$ is orthogonal to all classes in $\mathcal E_{M_{k-1}, K}$. In fact
$$\mathcal E_{M_{k-1}, K}=\{e\in \mathcal E_{M_{k}, K}|e\cdot E_k=0\}.$$

Suppose  $u$ is in the interior of the face $F_{E_k}$. Then  $u$ is positive on  $\mathcal E_{M_{k-1}, K}\subset
\mathcal E_{M_{k}, K}$. If we consider the expansion of $u$ with respect to 
the standard basis,  then $u$ has no $E_k$ coefficient. 
Hence $F_{E_k}\subset P_{M_{k-1}, K}$.

Conversely, if $u \in P_{M_{k-1}, K}$, then $u$ is orthogonal to $E_k$, and
$u$ is positive on $\mathcal E_{M_{k-1}, K}$. Moreover, for any class $e\in \mathcal E_{M_{k}, K}$
with nonzero $E_k$ coefficient, $u\cdot e= u(e+ (e\cdot E_k) E_k)$. Notice that  $u^2>0$ and
$$(e+(e\cdot E_k) E_k)^2= -1+(e\cdot E_k)^2\geq 0.$$
Further, since $e\cdot E_k>0$,  $u$ and $e+ (e\cdot E_k) E_k$ pairs positively with any symplectic form. By the light cone lemma, we have
$u\cdot e>0$. This proves that  $P_{M_{k-1}, K}\subset F_{E_k}$.
\end{proof}

\subsubsection{Proof of Proposition \ref{S=P}}


 \begin{proof} 
First of all,  $$\mathcal S_K^+\subset P_K.$$
This follows from \eqref{Kcone'}, \eqref{p-k}, the positive pairing  between $S_K^+$ and $\mathcal E_K$ by Lemma \ref{pairS+} (2), 
and the positive pairing between $S_K^+$ and  $K$ by the adjunction formula.

So now we start to prove $P_K\subset \mathcal S_K^+$.

For $k=1$, it is clear and is essentially contained in the proof of Theorem \ref{s2bundle-2}.
 
For   $2\le k\leq 8$, we do induction.
Suppose we have done the case when $k<l\le 8$, we want to argue that 
for $M_l=\mathbb C \mathbb P^2 \#l \overline{\mathbb C \mathbb P^2}$, $\mathcal S_K^+=P_K$. 

In this case, by Lemma \ref{face} (1),  $P_K$ is an open polytope with each face of the boundary a wall
 of a class in $\mathcal E_{M_l,K}$. Hence $e$ could be written as a finite 
combination $\sum_{i=1}^qa_ie_i$ with  $e_i$ in a boundary face and $a_i>0$. 
Notice each boundary face $F_{E'_i}$ (with $E'_i\in \mathcal E_{M_l, K}$) of $P_{M_l, K}$
 corresponds to $P_{M_{l-1}, K}$ by Lemma \ref{face} (2). Then by induction assumption, 
each $e_i\in \mathcal S_{M_{l-1}, K}^+$. Hence $e\in \mathcal S_K^+$ as well by definition.

When $k\ge 9$, we still do induction. However, in this situation, 
$P_K$ does have a wall contributed by the class $-K$. 

Consider a class of the form $V_a=aH-\sum_{i=1}^l E_i$. 
Given $e\in P_K$, 
we can find  $a<3$ such that $e\cdot (aH-\sum E_i)=0$. 
This is because $V_3=-K$ pairs positively
with $e$, and when $a=0$, $V_0$ pairs negatively with $e$.

Notice that
$V_a\cdot V_3<V_3\cdot V_3\leq 0$, so 
 the hypersurface of $V_a$  
does not intersect the wall of $-K$. Choose a generic line $L$ in this hypersurface 
such that $e\in L$ and $L$ 
intersects the polytope inside the interior of  the boundary faces $F_1, F_2$ at $e_1$ and $e_2$.

Then $e=a_1e_1+a_2e_2$ with $a_i>0$. 
Each class $e_i$ lies in the interior of the
 face $F_i$, which corresponds to $P_{M_{l-1}, K}$ by Lemma \ref{face} (2). By induction assumption, $e_i\in \mathcal S_K^+$. This finishes the proof.
\end{proof}

\subsection{Remark on the connection with Kodaira embedding} 
 Finally we would like to provide a heuristic comparison of the  genus zero subvariety-current-form construction 
with the Kodaira embedding theorem. 

Unlike linear systems in algebraic geometry,  the moduli spaces of pseudo-holomorphic subvarieties generally 
have no natural  linear structure. 
In algebraic geometry, we obtain a linear system of divisors from the vector space of  sections of a holomorphic line bundle. 
For almost complex structures, there is no such   direct global correspondence. In general, only locally  there are some traces
of  linearity.
Via the implicit function theorem, Taubes in \cite{T1} uses the linear structure on the kernel of the normal operator $D_C$ to linearize  the moduli space near  a subvariety $C$.  Another instance is that, inside a very  small ball, subvarieties behave like lines in the standard $\mathbb C^2$
(Lemma \ref{curves in a ball}).

However, it strikes us again and again  that the moduli spaces of genus zero  pseudo-holomorphic subvarieties  possess   local as well as semi-global linearity properties.
  Notably, in Proposition \ref{gendirect} we show that there are plenty of pencils of genus zero subvarieties in a big $J-$nef class $e$. 
To find a pencil,  we usually fix $e\cdot e$ points suitably, then look at the local or global moduli space passing through these points. 
We are thus led to the following global question. 

 \begin{question} \label{proj-space}Suppose $J$ is a tamed almost complex structure on a rational manifold, and $e$ is represented by a smooth $J-$holomorphic 
 genus zero
 subvariety. 
 Does $\mathcal M_e$ 
 have the structure of  a complex projective space of dimension $e\cdot e+1$?
 \end{question}

In fact, this has a positive answer when $e\cdot e\leq 0$, and when $e$ is the line class of $\mathbb {CP}^2$.
If we can answer positively Question \ref {proj-space} in general, 
it is tempting to further link  the construction of an almost K\"ahler form via genus zero subvarieties on a rational manifold to 
the construction of a K\"ahler form on a Hodge manifold via Kodaira embedding. Given a  $J-$ample  class $e$ in $S_{K_J}$, when $J$ is integrable and hence K\"ahler, 
holomorphic sections of $L$ provide a holomorphic embedding $\tau_L:M\rightarrow \mathbb{CP}^N$, where $L$ is  line bundle corresponding to the class $e$ and $N=h^0(L)-1=e\cdot e+1$. The pullback 
$\tau_L^*\omega_{FS}$ of the standard Fubini-Study form gives us a K\"ahler form in the class $e$. 
When $J$ is not integrable, one could heuristically   think that the family of  $J-$tamed forms with dominating $J-$invariant part  in the construction
come from  a family of embeddings  $\tau_e(\epsilon): M\rightarrow \mathcal M_e=\mathbb{CP}^N$, which are  close to being holomorphic.


\end{document}